\documentclass[pdflatex,sn-mathphys-num]{sn-jnl}
\usepackage[T1]{fontenc}
\usepackage{lmodern}
\ifdefined\pdfgentounicode
  \input{glyphtounicode}
\fi
\usepackage{amsmath,amssymb,amsthm}
\usepackage{mathtools}
\usepackage{tikz}
\theoremstyle{thmstyleone}
\newtheorem{theorem}{Theorem}[section]
\newtheorem{proposition}[theorem]{Proposition}
\newtheorem{lemma}[theorem]{Lemma}
\newtheorem{corollary}[theorem]{Corollary}
\theoremstyle{thmstyletwo}
\newtheorem{remark}[theorem]{Remark}

\newcommand{\R}{\mathbb{R}}
\newcommand{\C}{\mathbb{C}}
\DeclareMathOperator{\supp}{supp}

\DeclareMathOperator{\dist}{dist}
\DeclareMathOperator{\Mag}{Mag}
\newcommand{\Dmx}[1]{D_{\max}(#1)}
\newcommand{\indicator}[1]{\mathbf1_{#1}}

\begin{document}

\title[Maximum diversity on disks and balls]{Maximum diversity and support
transitions on disks and three-dimensional balls}

\author*[1]{\fnm{Adam} \sur{Dajani}}\email{Adam.Dajani@lonestar.edu}
\email{adamdajani4@gmail.com}
\affil*[1]{\orgname{Lone Star College--Montgomery},
  \orgaddress{\street{3200 College Park Drive}, \city{Conroe},
  \state{TX}, \postcode{77384-4500}, \country{USA}}}

\abstract{\unboldmath
For the similarity kernel $e^{-|x-y|}$, we determine the probability
measure of maximum diversity on every planar disk and
three-dimensional ball, answering a question of Leinster and Roff.
In each dimension the measure is uniform on the boundary up to a
threshold. Above it, the support is a closed central ball and
the outer sphere, separated by a mass-free annulus. In dimensions
two and three, their conjecture of finitely many supporting spheres
therefore holds only up to the respective threshold. We obtain the
measure explicitly up to one scalar root in three dimensions and
through Fredholm integral equations and a
scalar contact equation in the plane.
Computations in dimensions four and five suggest that interior
supporting spheres appear one at a time.
}


\maketitle

\section{Introduction}

The maximum diversity of a compact metric space is the
largest diversity of a probability measure on it, for the similarity
kernel $e^{-d(x,y)}$ \cite{leinsterroff2021}. On a nonempty compact set
$E\subset\R^N$ with the Euclidean metric it is
\begin{equation}\label{eq:divintro}
\Dmx{E}=
\left(\min_{\substack{\mu\ge0,\ \mu(E)=1\\ \supp\mu\subset E}}
\iint e^{-|x-y|}\,d\mu(x)\,d\mu(y)\right)^{-1}.
\end{equation}
The reciprocal of the double integral is the diversity of order two of
$\mu$. A minimizer maximizes the diversity of every order
$q\in[0,\infty]$ at once, and the maximum value is independent of $q$
\cite[Theorem~7.1 and Corollary~7.3]{leinsterroff2021}, as Leinster and
Meckes first proved for finite spaces \cite{leinstermeckes2016}. Its
logarithm is the maximum entropy of $E$
\cite[Definition~7.2]{leinsterroff2021}.

Rotational symmetry on a ball leaves the distribution of mass over the
radii to be found. Leinster and Roff ask for the maximizing measure on a
ball and conjecture that its support is a finite union of concentric
spheres, their number depending on the
radius
\cite[\S10, Q~(2)]{leinsterroff2021}.\footnote{Their interval
example settles the one-dimensional case, where the maximizing measure
has full support \cite[Example~8.4]{leinsterroff2021}.}
Bouafia also says that the support on Euclidean balls is still unknown
\cite{bouafia2026}. Leinster reiterated the problem in 2024, as a possible application of spectral geometry and semiclassical analysis, reporting Gimperlein's suggestion that spectral geometry could advance the study of magnitude \cite{leinster2024magnitude}.
In this paper we determine the maximizing measure for every radius in dimensions two and three using a potential-theoretic approach.

\begin{theorem}\label{thm:summary}
Let $N\in\{2,3\}$, let $\mu_R$ be the unique maximizing measure
for $e^{-|x-y|}$ on $\overline B_R\subset\R^N$, and let $\sigma_R$ be
normalized surface measure on $\partial B_R$. Let $R_N$ be the unique
$R>0$ at which the potential of $\sigma_R$ at the center equals its
value on the sphere,
\[
e^{-R}=\int e^{-|Re_1-y|}\,d\sigma_R(y).
\]
Thus $R_2\approx1.364$ and $R_3\approx2.729$.
If $0<R\le R_N$, then $\mu_R=\sigma_R$, and $\Dmx{\overline B_R}$ is the
reciprocal of the boundary potential in the displayed equation.
If $R>R_N$, there is $a\in(0,R)$ such that
\[
\supp\mu_R=\overline B_a\cup\partial B_R.
\]
The measure has positive density throughout $B_a$ and positive mass on
$\partial B_R$. Its mass on $\partial B_a$ is positive in dimension
three and zero in dimension two. The open annulus $a<|x|<R$ has zero
mass, and there the potential $e^{-|\cdot|}*\mu_R$ exceeds its value
$1/\Dmx{\overline B_R}$ on the support.
\end{theorem}

Positive density on the central ball rules out a finite union of
supporting spheres for $R>R_N$. The theorem follows from Section~\ref{sec:support} and the
pointwise positivity proved in Theorem~\ref{thm:disk}.

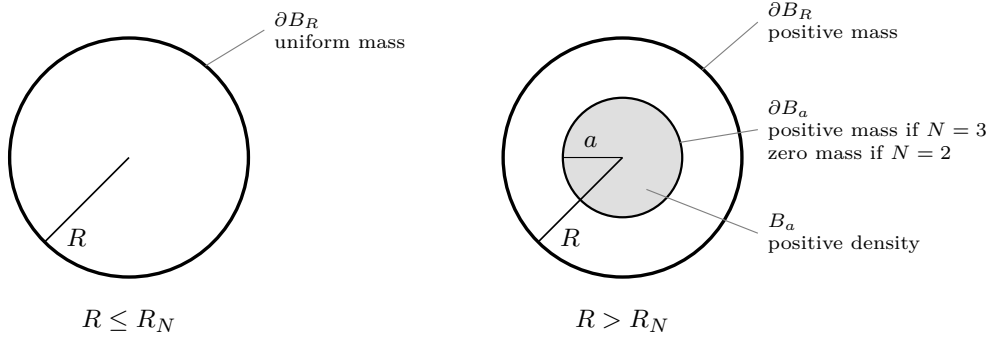
\begin{figure}[h]
\centering
\resizebox{\linewidth}{!}{%
\begin{tikzpicture}
[font=\small,
callout/.style=
{anchor=west,align=left,font=\footnotesize}]
\begin{scope}
\draw[line width=1.2pt] (0,0) circle (1.5);
\draw[line width=0.6pt] (0,0) -- (-1.07,-1.07);
\node at (-0.66,-1) {$R$};
\draw[gray] (0.96,1.15) -- (1.7,1.6);
\node[callout] at (1.7,1.6) {$\partial B_R$\\uniform mass};
\node at (0,-2.05) {$R\le R_N$};
\end{scope}
\begin{scope}[xshift=6.2cm]
\fill[gray!25] (0,0) circle (0.75);
\draw[line width=0.8pt] (0,0) circle (0.75);
\draw[line width=1.2pt] (0,0) circle (1.5);
\draw[line width=0.4pt] (0,0) -- (-0.75,0);
\node at (-0.4,0.2) {$a$};
\draw[line width=0.6pt] (0,0) -- (-1.05,-1.05);
\node at (-0.66,-1) {$R$};
\draw[gray] (1,1.1) -- (1.7,1.7);
\node[callout] at (1.7,1.7) {$\partial B_R$\\positive mass};
\draw[gray] (0.72,0.19) -- (1.7,0.35);
\node[callout] at (1.7,0.35)
{$\partial B_a$\\positive mass if $N=3$\\zero mass if $N=2$};
\draw[gray] (0.3,-0.4) -- (1.7,-0.95);
\node[callout] at (1.7,-0.95) {$B_a$\\positive density};
\node at (0,-2.05) {$R>R_N$};
\end{scope}
\end{tikzpicture}%
}
\caption{Support of the maximizing measure for $R\le R_N$
and $R>R_N$. At $R=R_N$ the center is also a contact point,
but has zero mass.}\label{fig:support}
\end{figure}

Up to the threshold the maximum diversity of the ball is therefore that of its boundary sphere, which is the magnitude of the sphere
\cite[Example~8.5]{leinsterroff2021}. Willerton computed this magnitude
for the two-sphere with the Euclidean metric \cite[Theorem~12]{willerton2014}.

Above the threshold in dimension three, Theorem~\ref{thm:ball} expresses the measure and
maximum diversity in terms of the unique root of a scalar equation
for the annulus width. For large $R$ the leading term is $R^3/6$,
and the $R^2$ coefficient in \eqref{eq:largeR} is approximately $0.928$.
Barcel\'o and Carbery's magnitude formula is
$R^3/6+R^2+2R+1$ \cite[Theorem~2]{barcelocarbery2018}.

In the plane the interior density is real analytic, and
Theorem~\ref{thm:disk} determines it through integral
equations. The density diverges like the inverse square root of the
distance to its edge, and the inner radius grows linearly from zero at the
threshold. At the threshold, the maximum diversity departs from the magnitude
of the circle $\partial B_R$ only at second order.

For $N=4,5$, numerical evidence points to a
further range of radii with finitely many supporting spheres
(Appendix~\ref{app:tables}).

For a probability measure $\mu$ on $E$ put $U_\mu=e^{-|\cdot|}*\mu$ and
$I(\mu)=\int U_\mu\,d\mu$. We call a minimizer of $I$ over these
probability measures an equilibrium measure. It is the maximizing
measure, whose uniqueness was proved by Meckes
\cite[Proposition~8.8]{leinsterroff2021}. Its
equilibrium level $\gamma=I(\mu)$ equals $1/\Dmx{E}$, and the potential
$U_\mu$ is their typicality function
\cite[Definition~3.4]{leinsterroff2021}. A maximizing measure is
balanced, with $U_\mu=\gamma$ on its support
\cite[Proposition~6.3]{leinsterroff2021}, and $U_\mu\ge\gamma$ on $E$
\cite[Lemma~7.7(i)]{leinsterroff2021}. We call the latter the
equilibrium inequality, and $\{x\in E:U_\mu(x)=\gamma\}$ the
\emph{contact set}.

The support proof is the same for both ball problems. For the
equilibrium measure $\mu$ of a ball put $U=U_\mu$ and
$J=(1-\Delta)U=(N-1)(e^{-|\cdot|}/|\cdot|)*\mu$.
At an interior contact point $U$ has a local minimum, so formally
$\Delta U\ge0$ and $J=U-\Delta U\le\gamma$; the latter inequality
is proved in Lemma~\ref{lem:contact}.
In dimensions two and three $J$ is
strictly subharmonic on every open set disjoint from the support, so
the maximum principle, applied to $J$ and then to $U-\gamma$, excludes
a mass-free annulus between interior contact radii, and also a
mass-free central ball bounded by an interior contact sphere. An open
annulus next to the outer sphere is always mass-free, and together these force the
supercritical support and contact set in Proposition~\ref{prop:support} (Fig.~\ref{fig:support}).

To recover the density we invert convolution with $e^{-|\cdot|}$. In $\R^3$ the
identity $(1-\Delta)^2U=8\pi\mu$ is local and determines the constant density
on the central ball. In $\R^2$ the inverse is proportional to
$(1-\Delta)^{3/2}$, and the density depends on the potential outside
the central disk as well; Section~\ref{sec:disk} obtains it through a screened
charged disk problem.

Meckes's variational formula for the magnitude $\Mag(E)$, denoted $|E|$
by Leinster and Roff, is a supremum over signed measures; restricted to
probability measures it is equivalent to maximum diversity, so
$\Dmx{E}\le\Mag(E)$ \cite{meckes2013}, \cite[(14)]{leinsterroff2021}.
Meckes also expresses maximum diversity as a constant multiple of
the Bessel capacity of order $(N+1)/2$ \cite{meckes2015}. In that form
\[
\Dmx{E}=\frac1{c_N}\inf\Bigl\{\bigl\|(1-\Delta)^{(N+1)/4}f\bigr\|_{L^2}^2:
f\in H^{(N+1)/2}(\R^N),\ f\ge1\hbox{ on }E\Bigr\},
\]
with $c_N$ the constant in \eqref{eq:fourier}, $N!$ times the
volume of the unit ball. Maximum diversity is thus an obstacle problem for a higher-order
operator, which is a parallel Barcel\'o and Carbery point out
\cite[\S10]{barcelocarbery2018}. Its minimizer is $U_\mu/\gamma$,
with coincidence set equal to the contact set above.

Apart from finite spaces, the maximizing measures
Leinster and Roff compute come from positive weight measures \cite[Lemma~8.2]{leinsterroff2021}.
Three-dimensional Euclidean balls have no weight measure
\cite{leinstermeckes2017survey}, so their maximizing measure misses
part of the ball \cite[\S8]{leinsterroff2021}.
For $R>R_N$ the measure $\mu_R/\gamma$ is a
positive weight measure on its support $\overline B_a\cup\partial B_R$
\cite[Lemma~5.6]{leinsterroff2021}. By their Corollary~7.4(ii),
$\Dmx{\overline B_R}$ equals the magnitude of that support. So
Theorems~\ref{thm:ball} and~\ref{thm:disk} compute the
magnitude of sets with a positive weight measure, a class they find
apparently rare \cite[\S8]{leinsterroff2021}. Leinster and Roff's uniform measure on a
ball is normalized Lebesgue measure \cite[Theorem~9.9]{leinsterroff2021}.
In three dimensions the density is constant on the central ball, whose
radius is $R-O(1)$ and which contains all but $O(1/R)$ of the mass
(Remark~\ref{rem:limits}), so after rescaling to the unit ball the
maximizing measure approaches it in total variation at the rate
$O(1/R)$.

Spherical to volume transitions also occur for Riesz energies with
radial external fields \cite{cmsvw2024,csw2023}. The volume part in
\cite{csw2023} reaches the sphere; the exponential equilibrium on a
ball has an intervening mass-free annulus between these components.

\section{Kernel identities}\label{sec:kernels}

For $N\ge2$ put $K(x)=e^{-|x|}$ and $L(x)=e^{-|x|}/|x|$ for $x\ne0$;
when it occurs in a potential, the locally integrable function $L$ is
given the value $+\infty$ at zero. We use the Fourier transform
$\widehat h(\xi)=\int e^{-ix\cdot\xi}h(x)\,dx$.
By radial differentiation,
\begin{equation}\label{eq:kernels}
(1-\Delta)K=(N-1)L
\end{equation}
and, away from the origin,
\begin{equation}\label{eq:laplacian}
\Delta L(x)=\frac{e^{-|x|}}{|x|^3}
\bigl(|x|^2+(3-N)|x|+3-N\bigr).
\end{equation}
The first identity holds as distributions on $\R^N$; its classical
second derivatives are locally integrable, and the flux of the bounded
gradient of $K$ across a sphere of radius $\varepsilon$ tends to zero.
For $N=2,3$ the second expression is positive.

For $\alpha,\zeta>0$,
\begin{equation*}\tag{G}\label{eq:gauss}
\int_0^\infty t^{-1/2}e^{-\alpha t-\zeta/t}\,dt
=\sqrt{\frac\pi\alpha}\,e^{-2\sqrt{\alpha\zeta}},
\qquad
\int_0^\infty t^{-3/2}e^{-\alpha t-\zeta/t}\,dt
=\sqrt{\frac\pi\zeta}\,e^{-2\sqrt{\alpha\zeta}},
\end{equation*}
and by monotone convergence the first holds at $\zeta=0$ and the second
at $\alpha=0$. The substitution $t\mapsto\zeta/(\alpha t)$ turns the
first integral into $\sqrt{\zeta/\alpha}$ times the second, and
$u=\sqrt{\alpha t}-\sqrt{\zeta/t}$ turns $\sqrt\alpha$ times the first
plus $\sqrt\zeta$ times the second into
$2e^{-2\sqrt{\alpha\zeta}}\int_{\R}e^{-u^2}\,du$. With $\alpha=|x|^2$ and
$\zeta=1/4$ the second formula expresses $K$ as a positive superposition
of the Gaussians $e^{-t|x|^2}$, and Fourier transformation of a Gaussian
yields
\begin{equation}\label{eq:fourier}
\widehat K(\xi)=\frac{c_N}{(1+|\xi|^2)^{(N+1)/2}},
\qquad c_N=2^N\pi^{(N-1)/2}\Gamma\left(\frac{N+1}{2}\right).
\end{equation}
By Legendre's duplication formula $c_N$ is $N!$ times the volume of the
unit ball, the constant in \cite[Proposition~9.7]{leinsterroff2021}; in
particular $c_3=8\pi$.
The multiplier is positive and integrable, so by Fourier inversion and
Fubini, for every finite signed measure $\eta$,
\[
I(\eta)=\iint K(x-y)\,d\eta(x)\,d\eta(y)
=\frac1{(2\pi)^N}\int_{\R^N}
\frac{c_N|\widehat\eta(\xi)|^2}{(1+|\xi|^2)^{(N+1)/2}}\,d\xi.
\]
Zero energy means that the continuous function $\widehat\eta$ vanishes
everywhere, and testing against Schwartz functions then shows $\eta=0$.
Strict positive definiteness gives uniqueness of an energy minimizer;
on a ball it is rotation invariant.
For finite measures $\mu,\nu$ of equal mass, with $U_\mu=K*\mu$,
\begin{equation}\label{eq:variation}
I(\nu)-I(\mu)=I(\nu-\mu)+2\int U_\mu\,d(\nu-\mu).
\end{equation}
Hence $U_\mu\ge c$ on $E$ together with $U_\mu=c$ on the support
suffices for optimality among probability measures on $E$. For a
minimizing probability measure the two conditions are necessary as
well, with $c=I(\mu)$, by first variation against point masses and
continuity of the potential.

For $N=3$, \eqref{eq:fourier} reads
$\widehat K(\xi)=8\pi(1+|\xi|^2)^{-2}$, so
\begin{equation}\label{eq:green}
(1-\Delta)^2K=8\pi\delta_0\quad\hbox{in }\R^3.
\end{equation}

In the plane the first formula in \eqref{eq:gauss}, with $\alpha=|x|^2$
and $\zeta=1/4$, expresses $\sqrt\pi L$ as a superposition of the same
Gaussians, and Fourier transformation followed by the second formula,
with $\alpha=0$ and $\zeta=(1+|\xi|^2)/4$, yields
\[
\widehat L(\xi)=\frac{2\pi}{\sqrt{1+|\xi|^2}},\qquad
\widehat K(\xi)=\frac{2\pi}{(1+|\xi|^2)^{3/2}}.
\]
For $N=2$, \eqref{eq:kernels} and \eqref{eq:laplacian} read
\begin{equation}\label{eq:plane}
(1-\Delta)K=L,\qquad
\Delta L(x)=e^{-|x|}\left(\frac1{|x|}+\frac1{|x|^2}
  +\frac1{|x|^3}\right)>0\quad(x\ne0).
\end{equation}

\section{Support on balls}\label{sec:support}

Let $\mu$ minimize the exponential energy among probability measures on
$\overline B_R\subset\R^N$, where $R>0$ and $N\ge2$, and set
\[
U=K*\mu,\qquad \gamma=I(\mu),\qquad J=(N-1)L*\mu.
\]
By \eqref{eq:kernels}, $(1-\Delta)U=J$ as distributions. The integral defining $J$ is
nonnegative and belongs to $L^1(\R^N)$, though individual values may at
first be infinite. By the first variation
\begin{equation}\label{eq:contact}
U\ge\gamma\quad\hbox{on }\overline B_R,
\qquad U=\gamma\quad\hbox{on }\supp\mu.
\end{equation}
By uniqueness $\mu$ is rotation invariant. Denote its pushforward under
$x\mapsto|x|$ by $\mu_{\rm rad}$, put $S=\supp\mu_{\rm rad}$, and for
$t>0$ let $\sigma_t$ be normalized surface measure on $\partial B_t$. Then
\begin{equation}\label{eq:radial}
\mu=\mu_{\rm rad}(\{0\})\delta_0
  +\int_{(0,R]}\sigma_t\,d\mu_{\rm rad}(t),
\qquad U(r)=\gamma\quad(r\in S).
\end{equation}

\subsection{Boundary measure in dimension three}

Let $\sigma_R$ be normalized surface measure on
$\partial B_R\subset\R^3$, and put $U_\partial=K*\sigma_R$.
Define
\begin{equation}\label{eq:Upsilon}
\Upsilon(R)=2R^2e^{-R}+(1+2R)e^{-2R}-1.
\end{equation}

\begin{lemma}\label{lem:sphere}
For $R>0$,
\begin{equation}\label{eq:sphere}
U_\partial(r)=\frac{e^{-R}}R\left((1+R)\frac{\sinh r}{r}-\cosh r\right)
\quad(0<r\le R).
\end{equation}
Its center value is $e^{-R}$, and its boundary value is
\begin{equation}\label{eq:bR}
U_\partial(R)=\frac{1-(1+2R)e^{-2R}}{2R^2}.
\end{equation}
If $R\le2$, $U_\partial$ is strictly decreasing on $[0,R]$.
If $R>2$, it is strictly increasing up to a single point
$r_R\in(0,R)$ and strictly decreasing thereafter. Consequently,
\begin{equation}\label{eq:endpoints}
U_\partial(r)>\min\{U_\partial(0),U_\partial(R)\}\quad(0<r<R).
\end{equation}
\end{lemma}

\begin{proof}
For $0<r\le R$ the angle between $re_1$ and a
$\sigma_R$-distributed point has cosine uniform on $[-1,1]$, and
after changing variables to the distance $u$ between the points,
\[
U_\partial(r)=\frac1{2rR}\int_{R-r}^{R+r}ue^{-u}\,du,
\]
which is \eqref{eq:sphere}.
The center value follows by continuity, and \eqref{eq:bR} is the
case $r=R$.

For the extrema introduce
$\ell(r)=(r\coth r-1)/r^2$ for $r>0$. Differentiating,
\[
\begin{aligned}
U_\partial'(r)&=\frac{e^{-R}}R\sinh r\bigl((1+R)\ell(r)-1\bigr),
\qquad \ell(0^+)=\tfrac13,\\
\ell'(r)&=-\frac1{4r^3\sinh^2r}
\sum_{n\ge3}\frac{(2n-4)(2r)^{2n}}{(2n)!}<0.
\end{aligned}
\]
Thus $(1+R)\ell(r)-1$ decreases strictly from $(R-2)/3$, and
\[
R^2\bigl((1+R)\ell(R)-1\bigr)
=\frac{2R(1+R)}{e^{2R}-1}-1<0,
\]
since $e^{2R}-1>2R+2R^2$.
Hence for $R\le2$ we have $U_\partial'<0$ throughout $(0,R]$, while for
$R>2$ there is one critical point $r_R\in(0,R)$,
with the asserted monotonicity on either side.
On the increasing part $U_\partial(r)>U_\partial(0)$; on the decreasing
part with $r<R$ we have $U_\partial(r)>U_\partial(R)$, and
\eqref{eq:endpoints} follows.
\end{proof}

\begin{proposition}\label{prop:sphere}
The function $\Upsilon$ has a unique positive zero
$R_3$, is positive on
$(0,R_3)$, and is negative on $(R_3,\infty)$.
Normalized surface measure $\sigma_R$ is the unique equilibrium
probability measure on $\overline B_R\subset\R^3$ if and only if
$R\le R_3$. In that range its equilibrium level is $U_\partial(R)$.
For $R<R_3$, the contact set is $\partial B_R$;
at $R=R_3$, it is $\partial B_R\cup\{0\}$.
For $R>R_3$, the equilibrium has positive mass in $B_R$.
\end{proposition}

\begin{proof}
Differentiating,
\[
\Upsilon'(R)=2Re^{-R}\bigl(2-R-2e^{-R}\bigr).
\]
The last factor starts at zero with positive derivative and is strictly
concave, tending to $-\infty$, so it crosses zero once on the positive
axis; $\Upsilon$ therefore increases from zero and then decreases to
$-1$, which proves the assertions about $R_3$.
By Lemma~\ref{lem:sphere},
\begin{equation}\label{eq:crossing}
U_\partial(0)-U_\partial(R)=\frac{\Upsilon(R)}{2R^2}.
\end{equation}
If $R\le R_3$, this comparison and
\eqref{eq:endpoints} imply $U_\partial\ge U_\partial(R)$ on the ball,
strictly on $B_R\setminus\{0\}$, with equality at the center only at
$R=R_3$. Optimality and the stated contact set follow from
the variational criterion in Section~\ref{sec:kernels}.

For $R>R_3$ move a mass $s>0$ from the sphere to the center.
The energy change is
\[
I\bigl((1-s)\sigma_R+s\delta_0\bigr)-I(\sigma_R)
=s^2I(\delta_0-\sigma_R)+2s\bigl(U_\partial(0)-U_\partial(R)\bigr)<0
\]
for sufficiently small $s>0$, so $\sigma_R$ fails to minimize on the
ball. Among probabilities supported on the sphere it does minimize,
uniquely, because its potential is constant on the sphere and the
kernel is strictly positive definite. An equilibrium with zero mass in
$B_R$ would then equal $\sigma_R$, contradicting the perturbation.
\end{proof}

\subsection{Boundary measure in dimension two}

Here $N=2$. Put $U_\partial=K*\sigma_R$ and $j_R=L*\sigma_R$.

\begin{proposition}\label{prop:endpoints}
For every $R>0$,
$U_\partial(r)\ge\min\{U_\partial(0),U_\partial(R)\}$ on $[0,R]$.
\end{proposition}
\begin{proof}
The potentials are smooth in $B_R$. By
\eqref{eq:plane}, $(1-\Delta)U_\partial=j_R$ and
$(rj_R'(r))'=r\Delta j_R(r)>0$ for $0<r<R$, and smooth radiality
requires $j_R'(0)=0$, so $j_R'(r)>0$ throughout this interval. Put
$w=U_\partial'$ and differentiate the radial equation,
\[
w''+\frac1r w'-\left(1+\frac1{r^2}\right)w=-j_R'<0,
\quad 0<r<R,
\]
with $w(0)=0$.
Suppose $w$ were negative somewhere and later nonnegative. Then the
component of $\{w<0\}$ containing that point would have its right
endpoint at most $R$ and its left endpoint in $[0,R)$, with zero values at
both endpoints, the origin included if necessary, and $w$ would attain
a negative minimum at some $r_0\in(0,R)$, where the displayed lefthand
side would be positive. So once $w$ becomes negative it stays negative.
The kernel is continuous, so $U_\partial$ is continuous at $R$.
Its minimum on $[0,R]$ is therefore attained at an endpoint.
\end{proof}

\begin{proposition}\label{prop:circle}
There is a unique positive solution $R_2$ of
$e^{-R}=\frac2\pi\int_0^{\pi/2}e^{-2R\sin t}\,dt$.
Normalized
arc length $\sigma_R$ is the unique equilibrium probability
measure for $e^{-|x-y|}$ on $\overline B_R\subset\R^2$ if and only
if $0<R\le R_2$.
\end{proposition}
\begin{proof}
The potential is constant on the circle, with value
$U_\partial(R)$, and
\[
U_\partial(0)=e^{-R},\qquad
U_\partial(R)=\frac2\pi\int_0^{\pi/2}e^{-2R\sin t}\,dt
=I_0(2R)-\mathbf L_0(2R).
\]
Here $I_0$ is the modified Bessel function of the first kind and order
zero and $\mathbf L_0$ the modified Struve function of order zero
\cite[(11.2.6), (11.5.5)]{dlmf}.
By Proposition~\ref{prop:endpoints} the equilibrium inequality is equivalent
to $U_\partial(0)\ge U_\partial(R)$. To locate the threshold put
\[
\chi(R)=e^R U_\partial(R)=\frac2\pi\int_0^{\pi/2}e^{R(1-2\sin t)}\,dt.
\]
Then $\chi(0)=1$, $\chi'(0)=1-4/\pi<0$, and $\chi''>0$ on $[0,\infty)$,
while the integral over $0\le t\le\pi/12$ tends to infinity, so
$\chi(R)\to\infty$; strict convexity forces a unique positive crossing
of $\chi=1$, and $\chi(R)\le1$ if and only if $R\in[0,R_2]$.
The variational criterion of Section~\ref{sec:kernels}
finishes the proof.
\end{proof}

\subsection{Interior contact estimates}

\begin{remark}\label{rem:process}
The conjecture suggested examining what a finite
collection of supporting spheres would leave empty. Once a central
atom is excluded such a configuration leaves a mass-free ball
around the origin. We sought a maximum principle argument to test
whether this is compatible with equilibrium. The auxiliary potential $J=(1-\Delta)U$ gives a second-order
argument for the support without a formula for the measure. Pursuing that avenue requires the interior contact bound $J \leq \gamma$, which Lemma~\ref{lem:contact} proves through second differences, not needing to assume the measure has a density.
\end{remark}

\begin{lemma}\label{lem:contact}
For every $N\ge2$ and every $x\in B_R$ with $U(x)=\gamma$, one has
$J(x)\le\gamma$. The measure $\mu$ is atomless. When $N=2$, it also
satisfies $\mu_{\rm rad}(\{t\})=0$ for every $0<t<R$.
\end{lemma}
\begin{proof}
For $z\ne0$ the eigenvalues of $D^2K(z)$ are $e^{-|z|}$ in the radial
direction and $-e^{-|z|}/|z|$ in the tangential directions, hence
\begin{equation}\label{eq:semiconcavity}
\delta^2_{h,e}K(z):=K(z+he)+K(z-he)-2K(z)\le h^2
\end{equation}
for every unit vector $e$; on a line through zero the same bound
holds, the first derivative having a downward jump there. With the
coordinate vectors $e_1,\ldots,e_N$ put
\[
A_h(z)=N-h^{-2}\sum_{i=1}^N\delta^2_{h,e_i}K(z)\ge0.
\]
At an interior contact point $\delta^2_{h,e_i}U(x)\ge0$ for all
sufficiently small $h$, so $\int A_h(x-y)\,d\mu(y)\le N$.
For $z\ne0$ the limit is
\[
A_h(z)\longrightarrow N-K(z)+(N-1)L(z),
\]
and at $z=0$ it is $+\infty$. By Fatou's lemma
$N-U(x)+J(x)\le N$, and the bound follows.

An atom at zero would put zero in the contact set and make $J(0)$
infinite, and an atom at a nonzero point would, by rotation invariance,
come with arbitrarily many distinct atoms of the same positive mass.
Thus $\mu$ is atomless. In dimension two an interior radial atom would
make $J$ infinite on its supporting circle, the angular
integrand being comparable to $1/(t|\theta|)$ near coincidence.
\end{proof}

\subsection{Boundary mass and mass-free regions}

\begin{lemma}\label{lem:boundary}
For every $N\ge2$, one has $U\in C^1(\R^N)$, $R\in S$ and $U'(R)<0$.
There is $\varepsilon>0$ such that $U(r)>\gamma$ for
$R-\varepsilon<r<R$ and $S\cap(R-\varepsilon,R)=\varnothing$.
In particular $q:=\mu_{\rm rad}(\{R\})>0$.
\end{lemma}
\begin{proof}
Since $K$ is Lipschitz and $\mu$ is atomless, dominated convergence
applies to the difference quotients, so
\begin{equation}\label{eq:gradient}
\nabla U(x)=-\int e^{-|x-y|}\frac{x-y}{|x-y|}\,d\mu(y),
\end{equation}
and dominated convergence once more makes it continuous.
If $t_0=\max S<R$, every source point has $|y|\le t_0$, so
$U'(r)<0$ for $r>t_0$ and $U(R)<U(t_0)=\gamma$, contradicting
\eqref{eq:contact}. So $R\in S$ and $U(R)=\gamma$.
At $r=R$ the radial derivative is
\[
U'(R)=-\int e^{-|Re_1-y|}\frac{R-y_1}{|Re_1-y|}\,d\mu(y)<0.
\]
The numerator is positive except at the single point $Re_1$, which
has zero mass. By continuity of $U'$, $U(r)>\gamma$ just inside $R$,
and \eqref{eq:contact} excludes those radii from $S$.
The radius $R$ is isolated in the support of $\mu_{\rm rad}$, hence has positive mass.
\end{proof}

For $N=2,3$ and $t>0$ put
\[
j(r,t)=\int L(re_1-y)\,d\sigma_t(y),
\qquad j(r,0)=L(re_1).
\]
For $0<r<t$ strict subharmonicity and radiality give
$(r^{N-1}\partial_rj(r,t))'=r^{N-1}\Delta j(r,t)>0$, with
$\partial_rj(0,t)=0$, so $j(r,t)$ increases strictly for $r<t$;
for $r>t$ every source distance increases with $r$ and $j(r,t)$
decreases strictly. In dimension two $j(t,t)=+\infty$.
In dimension three the distance substitution gives
\begin{equation}\label{eq:yukawa}
j(r,t)=\frac1{2rt}\int_{|r-t|}^{r+t}e^{-u}\,du
=\frac{e^{-|r-t|}-e^{-(r+t)}}{2rt}\quad(r,t>0).
\end{equation}
In particular it is continuous across $r=t$, with
$j(t,t)=(1-e^{-2t})/(2t^2)$, and the integral also shows
$j(r,t)\le1/\max\{r,t\}$, so by bounded convergence
$J(r)=2\int j(r,t)\,d\mu_{\rm rad}(t)$ is continuous at every positive
radius, including those with a sphere layer. For $N=2$ the infinite
one-sided limits come from monotone convergence outside the source
circle, or Fatou's lemma inside it. Hence the monotonicity bounds above
hold up to $r=t$ in both dimensions.

\begin{lemma}\label{lem:traces}
Let $N=2$ or $N=3$. Suppose $0<r_1<r_2<R$ are contact radii and
$\mu(\{r_1<|x|<r_2\})=0$. Then $J$ is smooth and strictly subharmonic
in this annulus and extends continuously from it to the boundary,
with limiting values $J(r_1),J(r_2)\le\gamma$.
If $0<r_1<R$ is a contact radius and $\mu(B_{r_1})=0$, the same conclusion
holds in $B_{r_1}$, with boundary value $J(r_1)\le\gamma$.
\end{lemma}
\begin{proof}
Either open region is locally separated from the source support, so
differentiation under the integral and \eqref{eq:laplacian}
imply smoothness and strict subharmonicity and Lemma~\ref{lem:contact}
bounds the endpoint values.

In the annulus split the radial integral at the source sets $t\le r_1$
and $t\ge r_2$. As $r\downarrow r_1$, by exterior monotonicity
$j(r,t)\le j(r_1,t)$ for $t\le r_1$, an upper bound integrable
because $J(r_1)<\infty$, and for each such $t$ the integrand converges to
$j(r_1,t)$; at $t=r_1$ that follows from
\eqref{eq:yukawa} when $N=3$, and when $N=2$ the contact
bound has already prevented a radial atom there. The sources $t\ge r_2$
stay uniformly separated from $r$ near $r_1$ and contribute continuously
by bounded convergence. Thus $J(r)\to J(r_1)$.

As $r\uparrow r_2$, by interior monotonicity
$j(r,t)\le j(r_2,t)$ for $t\ge r_2$, again an integrable bound, with
the same endpoint distinction at $t=r_2$ and the remaining sources
separated. Therefore $J(r)\to J(r_2)$.
For a mass-free central ball every source radius satisfies $t\ge r_1$, and
$j(r,t)\le j(r_1,t)$ for $r<r_1$. Dominated convergence gives
$J(r)\to J(r_1)$ as $r\uparrow r_1$. Local separation also makes
$J$ smooth at the center.
\end{proof}

\subsection{Analyticity of densities}

Constancy of $U$ determines $J$ on the planar disk but the density still depends on exterior values of $J$. Recovering it by an integral over the complement puts the diagonal singularity outside the domain of integration.

Take $N=2$ for the following recovery formula, and let
$\mathcal B=(1-\Delta)^{1/2}$ be the Fourier multiplier
$(1+|\xi|^2)^{1/2}$ on tempered distributions. From the transform of $L$,

\begin{equation}\label{eq:recovery}
\mathcal BJ=2\pi\mu.
\end{equation}
For $\tau>0$,
\[
\sqrt\tau=\frac1{2\sqrt\pi}\int_0^\infty(1-e^{-\tau t})t^{-3/2}\,dt.
\]
Integration by parts reduces it to the first formula in
\eqref{eq:gauss} with $\zeta=0$. Applied to the
Fourier multiplier, the same identity becomes
\begin{equation}\label{eq:semigroup}
\mathcal Bh=\frac1{2\sqrt\pi}\int_0^\infty
  (h-e^{-t}e^{t\Delta}h)t^{-3/2}\,dt,
\end{equation}
first for Schwartz functions and then by duality where used below,
and the positive off-diagonal density in the formula is
\begin{equation}\label{eq:levy}
n(z)=\frac1{8\pi\sqrt\pi}\int_0^\infty
  e^{-t-|z|^2/(4t)}t^{-5/2}\,dt
=\frac{(1+|z|)e^{-|z|}}{2\pi|z|^3}\quad(z\ne0).
\end{equation}
For the last equality differentiate the second formula in
\eqref{eq:gauss} in $\zeta$, then set $\alpha=1$ and $\zeta=|z|^2/4$.

\begin{lemma}\label{lem:analyticity}
Suppose $H\in L^1(\R^2)$ and $H=c$ almost everywhere in $B_a$, where $a>0$.
Then the distribution $\mathcal BH$ coincides in $B_a$ with the real analytic
function
\begin{equation}\label{eq:analyticity}
c+\int_{\R^2\setminus B_a}(c-H(y))n(x-y)\,dy.
\end{equation}
\end{lemma}

The proof is in Appendix~\ref{app:analyticity}.

For $N=2$, if $U=\gamma$ on a disk $B_a$ then
the identity $(1-\Delta)U=J$ implies $J=\gamma$ almost everywhere there, and by
Lemma~\ref{lem:analyticity} and \eqref{eq:recovery}
the restriction of $\mu$ to $B_a$ has the real analytic density
\begin{equation}\label{eq:density}
f(x)=\frac1{2\pi}\left[
\gamma+\frac1{2\pi}\int_{\R^2\setminus B_a}
(\gamma-J(y))\frac{(1+|x-y|)e^{-|x-y|}}{|x-y|^3}\,dy\right].
\end{equation}
It is nonnegative because $\mu$ is positive. Rotation invariance makes
it radial.

\subsection{Support classification}

\begin{lemma}\label{lem:regularity}
For $t>0$ put $p(r,t)=(K*\sigma_t)(re_1)$ in dimension three.
The function $p(\cdot,t)$ is $C^2$ across $r=t$, and
\[
\partial_r^3p(t+,t)-\partial_r^3p(t-,t)=\frac2{t^2}.
\]
For $a>0$, the radial potential of $\indicator{B_a}\,dx$
is $C^3$ near $r=a$.
\end{lemma}

\begin{proof}
The distance substitution gives
\[
rp(r,t)=\frac{(1+|r-t|)e^{-|r-t|}-(1+r+t)e^{-(r+t)}}{2t}.
\]
The function $v\mapsto(1+|v|)e^{-|v|}$ is $C^2$, and its third
derivative has jump $4$ at zero. Thus the third derivative of $rp(r,t)$
has jump $2/t$ at $r=t$. Since $p$ and its first two derivatives are
continuous there, the asserted jump is $2/t^2$.

For the volume potential, integrate $p(r,t)$ against $4\pi t^2\,dt$
over $0<t<a$. When $r$ is near $a$, the sources $t<a/4$ are separated
from $r$ and contribute smoothly. For $t\ge a/4$, the first three
radial derivatives are locally uniformly bounded, and the third
derivative is continuous in $r$ except at $t=r$. Dominated convergence
therefore makes the integrated third derivative continuous.
\end{proof}

\begin{proposition}\label{prop:support}
Let $N=2$ or $N=3$. Either $\mu=\sigma_R$, or there is $a\in(0,R)$ such
that
\[
S=[0,a]\cup\{R\},\qquad
\supp\mu=\overline B_a\cup\partial B_R,
\qquad U\equiv\gamma\quad\hbox{on }\overline B_a.
\]
In the latter case $\mu(\partial B_R)>0$ and $U(r)>\gamma$ for
$a<r<R$, so the contact set equals the support.
For $N=2$, the restriction to $B_a$ has a nonnegative real analytic
radial density, positive almost everywhere, and
$\mu(\partial B_a)=0$.
For $N=3$ the measure has the form
\begin{equation}\label{eq:ballform}
\mu=\frac{\gamma}{8\pi}\indicator{B_a}\,dx+m\sigma_a+q\sigma_R,
\quad m,q>0.
\end{equation}
\end{proposition}
\begin{proof}
Suppose that an open region is either a mass-free annulus between
two interior contact radii or a mass-free central ball bounded by
an interior contact sphere. By Lemma~\ref{lem:traces}, $J$
is smooth there, continuous up to the boundary and at most $\gamma$
on the boundary. The maximum principle gives $J<\gamma$ inside.
The function $u=U-\gamma$ is smooth and nonnegative in the region,
vanishes on its boundary and satisfies
\begin{equation}\label{eq:annulus}
(1-\Delta)u=J-\gamma<0.
\end{equation}
If $u$ is positive somewhere its positive interior maximum contradicts
the equation, and $u\equiv0$ contradicts it as well. Both regions
are therefore excluded.

By Lemma~\ref{lem:boundary}, $R$ is isolated in $S$. If
$S\setminus\{R\}$ is empty, $\mu=\sigma_R$ by rotation invariance.
Otherwise $S\setminus\{R\}$ is nonempty and compact. Its least element
is zero (a positive least element would bound a mass-free central ball) and its
greatest element $a$ is positive, as $a=0$ would make zero a point
atom. A point of $(0,a)$ outside $S$ would belong to a complementary
interval with endpoints in $S$, the left one positive, since a left
endpoint at zero would isolate zero in $S$ and make it an atom; both
endpoints would then be interior contact radii bounding a mass-free annulus,
and that case is already ruled out. So $S=[0,a]\cup\{R\}$, and
by \eqref{eq:contact} the potential equals $\gamma$ on
$\overline B_a$.

If an additional contact radius $r_2$ lay in $(a,R)$, then
$(a,r_2)$ would be a mass-free annulus between two interior contact
radii, which the first paragraph makes impossible. So $U>\gamma$ on
$a<r<R$. The contact set equals the support.

For $N=2$ Lemma~\ref{lem:contact} excludes mass on
$\partial B_a$, and formula~\eqref{eq:density} defines a
nonnegative real analytic radial density in $B_a$, with positive
integral over $B_a$ because the interior mass is positive and the
circle $\partial B_a$ carries zero mass. Analyticity then forces
positivity almost everywhere in $B_a$.

For $N=3$, $(1-\Delta)^2U=8\pi\mu$ implies
$\mu|_{B_a}=(\gamma/8\pi)\,dx$; the rest of the mass is carried by
$\partial B_a$ and $\partial B_R$, so by rotation invariance \eqref{eq:ballform} holds
with $m\ge0$ and $q>0$.
Put $s=r-a$ and $w(s)=r(U(r)/\gamma-1)$ in the annulus.
By Lemma~\ref{lem:regularity}, the inner sphere potential is $C^2$
with third-derivative jump $2/a^2$, and the volume potential is $C^3$
near $a$. The outer sphere potential is smooth there.
As $U=\gamma$ inside the central ball,
\[
w(0)=w'(0)=w''(0)=0,\qquad
w'''(0)=\frac{2m}{a\gamma},\qquad
(\partial_s^2-1)^2w=-(a+s).
\]
If $m=0$, evaluating the equation at zero yields
$w^{(4)}(0)=-a<0$, hence
$w(s)=-as^4/24+O(s^5)<0$ for sufficiently small $s>0$, against
$U\ge\gamma$ on the ball. So $m>0$.
\end{proof}

\begin{remark}\label{rem:subsets}
Let $F$ be compact with
$\supp\mu_R\subseteq F\subseteq\overline B_R$, equipped with the
induced Euclidean metric. Then $\mu_R$, regarded as a probability
measure on $F$, is its unique maximizing measure, and
$\Dmx{F}=\Dmx{\overline B_R}$.

The measure $\mu_R$ has mass one on $F$, and every probability on $F$
extends by zero to one on $\overline B_R$ with the same exponential energy.
Thus the two minimum energies coincide, and uniqueness on the ball
implies uniqueness on $F$. The diversity conclusions follow from
\eqref{eq:divintro} and \cite[Corollaries~7.3 and~7.5]{leinsterroff2021}.
\end{remark}

\section{Equilibrium on three-dimensional balls}\label{sec:ball}

Here $N=3$ now. Denote the equilibrium measure by $\mu_R$, with $U=K*\mu_R$ and $\gamma=1/\Dmx{\overline B_R}$ as in
Section~\ref{sec:support}. By \eqref{eq:green},
$(1-\Delta)^2U=8\pi\mu_R$. The function $\Upsilon$ is as in
\eqref{eq:Upsilon}. For $R > R_3$ we first rescale the measure to unit potential on $B_a$, so the volume density is $1/(8\pi)$. The sphere coefficients must cancel the $\cosh r$ and $\sinh r / r$ terms and after this cancellation, contact at the outer sphere leaves a single equation for the annulus width $b=R-a$.

For $b>0$ define
\begin{equation}\label{eq:PQ}
P(b)=b-\frac{b^2+b}{2}e^{-b}-\frac{\sinh b}{2},\qquad
Q(b)=b^2+be^{-b}-\sinh b.
\end{equation}
\begin{lemma}\label{lem:roots}
The function $P$ has a unique positive zero $\beta<R_3$.
It is positive below $\beta$ and negative above it.
Also $Q>0$ on $(0,R_3)$ and $Q(R_3)=0$.
\end{lemma}

\begin{proof}
Expanding $4e^bP(b)$ gives
\[
4e^bP(b)=4be^b-e^{2b}+1-2b-2b^2
=\frac23b^3-\sum_{n\ge5}\frac{(2^n-4n)b^n}{n!}.
\]
After division by $b^3$ what is left decreases strictly from $2/3$ to
$-\infty$, every subtracted coefficient being positive because
$2^n-4n>0$ for $n\ge5$, and the assertions about $P$ follow. Those
about $Q$ follow from the identity $Q(b)=e^b\Upsilon(b)/2$.
Finally,
\[
Q(b)-2P(b)=b\bigl(b-2+(b+2)e^{-b}\bigr).
\]
The bracket vanishes at zero and its derivative is $1-(b+1)e^{-b}>0$, so
$2P<Q$ on $(0,\infty)$, and $Q(R_3)=0$ forces
$P(R_3)<0$. So $\beta<R_3$.
\end{proof}

\begin{theorem}\label{thm:ball}
Let $R_3$ be the zero of $\Upsilon$ from
Proposition~\ref{prop:sphere}.
\begin{enumerate}

\item[(i)] If $R\le R_3$, then $\mu_R=\sigma_R$ and
\begin{equation}\label{eq:ballvalue}
\Dmx{\overline B_R}=\frac{2R^2}{1-(1+2R)e^{-2R}}.
\end{equation}
In particular $\Dmx{\overline B_{R_3}}=e^{R_3}$.
The contact set is $\partial B_R$ for $R<R_3$ and
$\partial B_R\cup\{0\}$ for $R=R_3$, and $U>\gamma$ at every
other point of $\overline B_R$.

\item[(ii)] If $R>R_3$, then $0<\beta<R_3$, and the
scalar equation
\begin{equation}\label{eq:root}
(R-b)P(b)+Q(b)=0
\end{equation}
has a unique solution $b\in(\beta,R_3)$. With
\begin{equation}\label{eq:masses}
a=R-b,\qquad g=a+1-\frac{a+2}{b},\qquad
M=\frac{a^3}{6}+\frac{ag}{2}+\frac{R(a+2)e^b}{2b},
\end{equation}
the maximizing measure and the maximum diversity are
\begin{equation}\label{eq:ballmeasure}
\mu_R=\frac1M\biggl(\frac1{8\pi}\indicator{B_a}\,dx
+\frac{ag}2\,\sigma_a+\frac{R(a+2)e^b}{2b}\,\sigma_R\biggr),
\qquad \Dmx{\overline B_R}=M.
\end{equation}
Both spherical masses are positive.
\end{enumerate}
\end{theorem}

\begin{proof}[Proof of Theorem~\ref{thm:ball}(i)]
Proposition~\ref{prop:sphere} determines the measure and the contact
set, and its level $U_\partial(R)$ in \eqref{eq:bR} is
$1/\Dmx{\overline B_R}$. At $R=R_3$,
\eqref{eq:crossing} shows
$U_\partial(R_3)=e^{-R_3}$.
\end{proof}

\subsection{Supercritical equilibrium}

Multiplying the potential excess by $r$ strips the annular equation of its variable coefficients and reduces the positivity question to a comparison of a homogeneous solution with particular solutions for constant and linear forcing.

\begin{lemma}\label{lem:width}
For $s>0$ put
\[
\begin{gathered}
k_0(s)=\frac{s\cosh s-\sinh s}{2},\qquad
k_1(s)=\frac{s\sinh s-2\cosh s+2}{2},\\
k_2(s)=\frac{s\cosh s-3\sinh s+2s}{2}.
\end{gathered}
\]
These functions are positive, and, for every $a\ge0$,
\[
f_a(s)=\frac{ak_1(s)+k_2(s)}{k_0(s)}
\]
is strictly increasing.
\end{lemma}

\begin{proof}
Since $k_0'(s)=s\sinh s/2>0$ and $k_0(0)=0$, positivity follows from
\[
k_1(s)=\int_0^s k_0(u)\,du,\qquad
k_2(s)=\int_0^s(s-u)k_0(u)\,du.
\]
The identity
\[
k_0'(s)^2-k_0(s)k_0''(s)=\frac{\sinh^2s-s^2}{4}>0
\]
makes $k_0'/k_0$ strictly decreasing, and for fixed $t>0$ and $s>t$
\[
\frac{d}{ds}\log\frac{k_0(s-t)}{k_0(s)}
=\frac{k_0'(s-t)}{k_0(s-t)}-\frac{k_0'(s)}{k_0(s)}>0.
\]
In each of $k_1(s)/k_0(s)=\int_0^s k_0(s-t)/k_0(s)\,dt$ and
$k_2(s)/k_0(s)=\int_0^s tk_0(s-t)/k_0(s)\,dt$, increasing the upper limit increases
the integrand on the common interval and adds a positive integral, so
both quotients are strictly increasing, and with them $f_a$ for every
$a\ge0$.
\end{proof}

\begin{proof}[Proof of Theorem~\ref{thm:ball}(ii)]
The left side of \eqref{eq:root} is
positive at $b=\beta$ and negative at $b=R_3$, so a root
lies between them; fix any root in this interval and set
$a=R-b$. The identities
\[
bk_1(b)-(b-1)k_0(b)=P(b),\qquad
bk_2(b)-(b-2)k_0(b)=Q(b)
\]
turn the root equation into
\begin{equation}\label{eq:g}
g=f_a(b)=a+1-\frac{a+2}{b}>0.
\end{equation}
Define $M$ by \eqref{eq:masses} and the positive measure
\[
\nu=\frac1{8\pi}\indicator{B_a}\,dx+\frac{ag}2\,\sigma_a
+\frac{R(a+2)e^b}{2b}\,\sigma_R,
\]
of total mass $M$. In this proof put $U=K*\nu$; after normalization,
the probability measure and its potential are $\nu/M$ and $U/M$.

For $t>0$ and $0\le r\le t$ the potential of normalized surface measure is
\begin{equation}\label{eq:shell}
p(r,t)=\frac{e^{-t}}{t}
\left((t+1)\frac{\sinh r}{r}-\cosh r\right),
\end{equation}
with the continuous value at $r=0$; for $r\ge t$ use $p(r,t)=p(t,r)$.
The formula follows from the distance substitution of
Lemma~\ref{lem:sphere}. Since $\int_{\R^3}e^{-|x|}\,dx=8\pi$ we find the volume
contribution by subtracting the contribution from radii $t>a$, which
for $0\le r\le a$ gives
\[
U_{\rm vol}(r)=1-\frac{e^{-a}}2
\left((a^2+3a+3)\frac{\sinh r}{r}-(a+1)\cosh r\right).
\]
Here the radial density $t^2/2$ produces tail integrals
$e^{-a}(a^2+3a+3)$ and $e^{-a}(a+1)$. Adding the two sphere potentials
cancels the coefficient of $\cosh r$, since
\[
\frac{e^{-a}}2\left(a+1-g-\frac{a+2}{b}\right)=0.
\]
Likewise for the coefficient of $\sinh r/r$,
\[
\frac{e^{-a}}2\left(-a^2-3a-3+g(a+1)
+\frac{(a+2)(R+1)}{b}\right)=0,
\]
by the definition of $g$ and $R=a+b$.. Thus $U=1$ on $\overline B_a$.

By
\eqref{eq:green}, $(1-\Delta)^2U=8\pi\nu$, and since $U$ is
radial, for $r>0$
\[
r(1-\Delta)^2U=\left(\frac{d^2}{dr^2}-1\right)^2(rU).
\]
Lemma~\ref{lem:regularity} gives the sphere and volume regularity
at $a$. With $s=r-a$ and $w(s)=r(U(r)-1)$,
\[
w(0)=w'(0)=w''(0)=0,\qquad w'''(0)=\frac2a\cdot\frac{ag}2=g,
\qquad
\left(\frac{d^2}{ds^2}-1\right)^2w=-(a+s).
\]
The solution is
\begin{equation}\label{eq:excess}
w(s)=gk_0(s)-ak_1(s)-k_2(s)
=k_0(s)\bigl(f_a(b)-f_a(s)\bigr).
\end{equation}
By Lemma~\ref{lem:width}, $w(s)>0$ for $0<s<b$ and $w(b)=0$.
Hence $U>1$ on the open annulus, with $U=1$ on both supporting components.

The probability measure $\nu/M$ has potential at least $1/M$, with
equality on its support, so by strict positive definiteness and
\eqref{eq:variation} it is the unique minimizer, of energy $1/M$.
Two roots of \eqref{eq:root} in $(\beta,R_3)$ would
produce, through the construction, different positive inner radii,
hence different supports, contradicting uniqueness. Every component mass is positive, so the
support is $\overline B_a\cup\partial B_R$, and by
\eqref{eq:excess} it equals the contact set.
\end{proof}

\begin{corollary}\label{cor:capacity}
Let $\lambda>0$. By \eqref{eq:green} and scaling,
$\kappa_\lambda(x)=e^{-\lambda|x|}/(8\pi\lambda)$ is the Green kernel of
$(\lambda^2-\Delta)^2$ on $\R^3$, the Mat\'ern kernel of smoothness
$\tfrac12$. Its equilibrium measure on $\overline B_R$ is the image of
$\mu_{\lambda R}$ under $x\mapsto x/\lambda$, and the capacity of
$\overline B_R$, the reciprocal of the least $\kappa_\lambda$-energy of a
probability measure on it, is $8\pi\lambda\,\Dmx{\overline B_{\lambda R}}$.
\end{corollary}

\begin{proof}
The dilation $x\mapsto\lambda x$ sends $\overline B_R$ onto
$\overline B_{\lambda R}$, and the $\kappa_\lambda$-energy of a
probability measure on $\overline B_R$ is $1/(8\pi\lambda)$ times the
exponential energy of its image. The claim follows on minimizing both sides.
\end{proof}

\subsection{Radius limits}

\begin{proposition}\label{prop:width}
The root $b=b(R)$ is continuous and strictly decreasing, with
\[
\lim_{R\downarrow R_3}b(R)=R_3,
\qquad \lim_{R\to\infty}b(R)=\beta.
\]
The maximum diversity is continuous at $R_3$.
\end{proposition}

\begin{proof}
Set $G_R(t)=(R-t)P(t)+Q(t)$. Continuity comes from taking
subsequential limits in \eqref{eq:root}; at a fixed
$R>R_3$ the strict endpoint signs rule out $\beta$ and
$R_3$ as a limiting root, and the unique interior root is
all that is left. For strict decrease take
$R'>R>R_3$. Then
\[
G_{R'}(b(R))=(R'-R)P(b(R))<0,
\qquad G_{R'}(\beta)=Q(\beta)>0.
\]
The unique root for $R'$ is therefore smaller than $b(R)$.
As $R\to\infty$, boundedness of $b$ implies $P(b)=-Q(b)/(R-b)=O(1/R)$,
hence $b\to\beta$. The series for $4e^bP(b)/b^3$ in
Lemma~\ref{lem:roots} has a negative termwise derivative, so
$P'(\beta)<0$ and $b-\beta=O(1/R)$.
Were the limit $b_0$ at the threshold smaller than $R_3$, we would
have $a_0=R_3-b_0>0$ and $a_0P(b_0)+Q(b_0)=0$, which are all
the conditions the construction with positive inner radius uses, and it would produce a
second minimizer at $R=R_3$, contrary to
Theorem~\ref{thm:ball}(i). So $b\to R_3$ and $a\to0$ at
the threshold. Substitution in \eqref{eq:masses} shows
$M\to e^{R_3}$, so the maximum diversity is continuous at
$R_3$.
\end{proof}

\begin{remark}\label{rem:limits}
For large $R$ both spherical masses in \eqref{eq:ballmeasure} are
$O(1/R)$, and collecting the terms of order $R^3$ and $R^2$ in $M$
gives
\begin{equation}\label{eq:largeR}
\Dmx{\overline B_R}=\frac{R^3}{6}
+\frac12\Bigl(\frac{e^\beta}{\beta}+1-\beta-\frac1\beta\Bigr)R^2+O(R),
\quad R\to\infty,
\end{equation}
with leading term as in \cite[Proposition~9.7]{leinsterroff2021}.
Meckes's variational formula and
\cite[Theorem~2]{barcelocarbery2018} also show
\[
\Dmx{\overline B_R}\le\Mag(\overline B_R)
=\frac{R^3}{6}+R^2+2R+1.
\]
The coefficient of $R^2$ in \eqref{eq:largeR} is
approximately $0.928$. For $R=3,4,\ldots,10$, Appendix~\ref{app:tables} lists
the annulus width $b$ from \eqref{eq:root} and the maximum
diversity from \eqref{eq:masses}, next to the magnitude.

At small radius, the formula in Theorem~\ref{thm:ball}(i) expands to
$\Dmx{\overline B_R}=1+\frac43R+O(R^2)$, so
$\Mag(\overline B_R)-\Dmx{\overline B_R}=\frac23R+O(R^2)$.
Roff and Willerton's microscopic weighting of the unit three-ball
includes a normal-derivative term on the boundary sphere
\cite[Proposition~6.8]{roffwillerton2026}.
That term is absent from the subthreshold maximizing probability.
\end{remark}

\section{Equilibrium on planar disks}\label{sec:disk}

In the plane $\widehat K=2\pi(1+|\xi|^2)^{-3/2}$, so recovering a
measure from its exponential potential takes
$(1-\Delta)^{3/2}=(1-\Delta)(1-\Delta)^{1/2}$, and the square root is
nonlocal. Convolution with $L/(2\pi)$ inverts the square root, since
$\widehat L=2\pi(1+|\xi|^2)^{-1/2}$, and $L$ is the trace on the plane
$z=0$ of the Yukawa kernel $e^{-|X|}/|X|$, which is $4\pi$ times the
fundamental solution of $1-\Delta_3$ on $\R^3$. For a minimizer $\mu$,
by \eqref{eq:plane}, $J=(1-\Delta)U$ equals $L*\mu$, the trace of the
Yukawa potential of $\mu$ regarded as a charge on that plane, and
$J=\gamma$ almost everywhere on every open set where $U=\gamma$. Above the
threshold the central disk is held at constant potential in the
field of the charged outer circle, with screening length one. Without
the screening and the circle it is the classical electrified disk, whose
density blows up like the inverse square root of the distance to the
edge. As $R$ decreases to the threshold, the normalized interior
charge, rescaled to the unit disk, tends to the equilibrium charge of
that disk (Theorem~\ref{thm:threshold}), and above it the
planar density obeys the same edge law (Theorem~\ref{thm:edge}).

\begin{lemma}\label{lem:trace}
Let $a>0$ and let $f$ be a measurable function on $B_a\subset\R^2$ with
\[
|f(y)|\le C(a^2-|y|^2)^{-1/2}.
\]
Its three-dimensional Yukawa potential
\[
\int_{B_a}
\frac{e^{-|X-(y,0)|}}{|X-(y,0)|}f(y)\,dy
\]
is continuous on $\R^3$ and tends to zero at infinity. Its trace on
$z=0$ is $L*(f\,dx)$, with equality at every point, including
$\partial B_a$.
\end{lemma}

\begin{proof}
Let $\nu$ be the measure $|f(y)|\,dy$ on the plane $z=0$. We first
prove the uniform estimate
\begin{equation}\label{eq:localmass}
\nu(B_s(X))\le C_a s^{3/2}
\quad(X\in\R^3,\ 0<s<a/8).
\end{equation}
Projecting to the plane reduces it to planar balls. In
$B_{a/2}$ the density is bounded and the contribution is $O(s^2)$.
For $a/2\le t<a$ a circle of radius $t$ meets a ball of radius
$s<a/8$ in arcs of total length at most a constant times $s$, and the
contributing radii stay in an interval $I$ of length at most $2s$, so, since
$a^2-t^2\ge a(a-t)$, polar integration bounds this part by
\[
C_a s\int_I(a-t)^{-1/2}\,dt\le C_a s^{3/2},
\]
uniformly over intervals $I\subset[a/2,a]$ of length at most $2s$.
Enlarging the constant to cover both parts leads to
\eqref{eq:localmass}.

Layer-cake integration then gives, uniformly in $X$ and for
$0<\delta<a/8$,
\begin{align}
\int_{|X-Y|<\delta}\frac{d\nu(Y)}{|X-Y|}
&=\frac{\nu(B_\delta(X))}{\delta}
  +\int_0^\delta\frac{\nu(B_s(X))}{s^2}\,ds
\le3C_a\delta^{1/2}.
\label{eq:nearfield}
\end{align}
Since the Yukawa kernel is bounded by the Coulomb kernel, removing its
singularity inside a ball of radius $\delta$ changes
the potential uniformly by $O(\delta^{1/2})$; with continuous
truncated kernels the potentials are continuous by dominated
convergence, hence so is their uniform limit. Decay at infinity comes
from compact support and integrability of $f$. The trace identity is
immediate from the defining integral, which is now finite at every
planar point.
\end{proof}

\begin{lemma}\label{lem:lift}
Let $0<a<R$ and let $\psi$ be a measurable radial function on $B_a$, with
$|\psi(y)|\le C(a^2-|y|^2)^{-1/2}$. Suppose that
$L*(\psi\,dx)=j_R=L*\sigma_R$ on $B_a$. Regard the disk and circle as subsets of
the plane $z=0$, and put
\[
W(X)=\int_{\partial B_R}
  \frac{e^{-|X-(y,0)|}}{|X-(y,0)|}\,d\sigma_R(y)
  -\int_{B_a}\frac{e^{-|X-(y,0)|}}{|X-(y,0)|}\psi(y)\,dy.
\]
Then $W=0$ continuously on the closed disk, $W>0$ on the complement
of the disk and circle, and $W(X)\to+\infty$ as $X$ approaches the
circle. Its planar trace $j_R-L*(\psi\,dx)$ is zero on
$\overline B_a$ and positive elsewhere, with value $+\infty$ on
$\partial B_R$. It belongs to $L^1(\R^2)$, with a logarithmic
singularity at $r=R$.
\end{lemma}

\begin{proof}
The circle potential is smooth near the closed disk
because $a<R$, and the assumed identity with Lemma~\ref{lem:trace}
shows that $W=0$ continuously on the whole closed disk. Near the circle the disk
potential stays bounded. After a rotation let $X=(r,0,z)$ and put
$d^2=(r-R)^2+z^2$; the squared distance to a source point at angle
$\theta$ is $d^2+4rR\sin^2(\theta/2)$, so near the circle the
circle potential is comparable to
\[
\int_{-\varepsilon}^{\varepsilon}(d^2+s^2)^{-1/2}\,ds,
\qquad d=\dist(X,\partial B_R),
\]
and tends to $+\infty$ as $d\downarrow0$.
Off the disk and circle $W$ is smooth with
$(1-\Delta_3)W=0$, and it tends to zero at infinity.
A negative value of $W$ would, with these boundary limits, force a
negative minimum at an interior point of that complement, where
$(1-\Delta_3)W<0$, a contradiction. The complement is connected. The strong
maximum principle and the positive circle singularity imply $W>0$
throughout it.

The asserted trace signs follow by restriction to the plane. For
$r$ near $R$ the squared distance from $re_1$ to a point at angle
$\theta$ on the circle is
$(r-R)^2+4rR\sin^2(\theta/2)$, with a local integral of logarithmic
growth, and the disk contribution is smooth near $r=R$.
Finally $L\in L^1(\R^2)$ and both source measures have finite total
variation, and Tonelli's theorem, applied to their absolute values, makes
the trace integrable.
\end{proof}

For the Coulomb kernel $\mathcal C(x)=|x|^{-1}$, put
$E_{\mathrm C}(\eta)=\iint\mathcal C(x-y)\,d\eta(x)\,d\eta(y)$
whenever the integral is defined.

\begin{lemma}\label{lem:coulomb}
The probability measure
\[
d\mu_{\mathrm C}(z)=\frac{\indicator{\{|z|<1\}}}{2\pi\sqrt{1-|z|^2}}\,dz
\]
has Coulomb potential $\pi/2$ on the closed unit disk. It uniquely
minimizes the energy with kernel $|x-y|^{-1}$ among probability
measures on that disk.
\end{lemma}

\begin{proof}
For the kernel $|x-y|^{-1}$ the density holding the unit disk at
unit potential is $1/(\pi^2\sqrt{1-|z|^2})$, counted on both faces
\cite[\S4(i)]{copson1947}, and dividing by its mass $2/\pi$ yields
$\mu_{\mathrm C}$, with Coulomb potential $\pi/2$ for
$0<|z|<1$. The density of $\mu_{\mathrm C}$ is at most
$(2\pi)^{-1}(1-|z|^2)^{-1/2}$, so the near-field bound
\eqref{eq:nearfield} and the truncation argument in the proof
of Lemma~\ref{lem:trace} make the potential continuous on
$\R^3$. Hence it equals $\pi/2$ on the closed disk.

The Coulomb kernel is strictly positive definite on finite-energy signed measures in $\R^2$
\cite[\S8.1]{fuglede1960}. A finite energy probability measure on a compact
set whose potential is at least its energy
throughout the set minimizes energy among probability measures on the
set, and strict positive definiteness makes the minimizer unique
\cite[Remark~2 after Theorem~2.4]{fuglede1960}. On the closed disk the
potential of $\mu_{\mathrm C}$ equals $\pi/2=E_{\mathrm C}(\mu_{\mathrm C})$.
\end{proof}

For the rest of this section take $N=2$ and minimize exponential
energy over probability measures on $\overline B_R$, with minimizer
$\mu_R$, $U=K*\mu_R$, $\gamma=\int U\,d\mu_R$
and $J=L*\mu_R$. Here $\sigma_R$ is normalized arclength and
$j_R=L*\sigma_R$.

For $R>R_2$, Proposition~\ref{prop:circle} excludes the
boundary regime. Denote the inner radius, interior density and outer
mass in Proposition~\ref{prop:support} by $a$, $f_R$ and $q$,
respectively. Contact on the central disk requires
\[
\int_{B_a}L(x-y)f_R(|y|)\,dy=\gamma-qj_R(x)
\quad\hbox{for almost every }x\in B_a.
\]

Copson solves the unscreened disk at prescribed potential by an Abel
construction \cite{copson1947}, and Agra, Trizac and Bocquet reduce the
screened constant-potential problem to a Fredholm equation by dual
integral methods, observing the same inverse-square-root edge law by a
local analysis \cite{agratrizacbocquet2004}.

The factor $1/\pi^2$ in $\mathcal T$ below makes
$\mathcal A_a\mathcal T1=1/(\pi^2\sqrt{a^2-r^2})$ the density holding
$B_a$ at unit Coulomb potential, as in the proof of
Lemma~\ref{lem:coulomb}.

\subsection{Abel inversion}\label{sec:abel}

For a trial inner radius $0<a<R$, consider
$L*(f\,dx)=v$ on $B_a$. Densities on $B_a$ are extended by zero
when forming whole-plane convolutions. The Abel inversion
separates the possible edge singularity from the regular unknown $g$, placing it in the
explicit term $g(a)/\sqrt{a^2- r^2}$. Once the Coulomb part has been inverted, screening
becomes an operator correction of norm less than one. The cosine kernel in the Abel reduction occurs in
\cite[eqs.~(36)--(38)]{seki2026}. For a smooth radial function $v$ on $\overline B_a$ define
\begin{equation}\label{eq:abeldata}
\begin{split}
(\mathcal Tv)(t)&=\frac1{\pi^2}\frac{d}{dt}
  \int_0^t\frac{r v(r)}{\sqrt{t^2-r^2}}\,dr,\\
h(x)&=\int_0^{\pi/2}\sin\theta\,e^{-|x|\sin\theta}\,d\theta,\\
(\mathcal P_ag)(t)&=\frac1\pi\int_0^a
  \bigl[h(t-s)+h(t+s)\bigr]g(s)\,ds.
\end{split}
\end{equation}
For $g\in C^2[0,a]$ with $g'(0)=0$, set
\begin{equation}\label{eq:abelinverse}
(\mathcal A_ag)(r)=\frac{g(a)}{\sqrt{a^2-r^2}}
-\int_r^a\frac{g'(t)}{\sqrt{t^2-r^2}}\,dt,
\quad 0\le r<a.
\end{equation}

\begin{lemma}\label{lem:abel}
The equation $(I-\mathcal P_a)g=\mathcal Tv$ has a unique continuous solution,
given by its uniformly convergent Neumann series. The operator $\mathcal P_a$ satisfies
\begin{equation}\label{eq:contraction}
\|\mathcal P_a\|_{C[0,a]\to C[0,a]}=\omega(a)
=1-\frac2\pi\int_0^{\pi/2}e^{-a\sin\theta}\,d\theta<1.
\end{equation}
The solution belongs to $C^2[0,a]$, satisfies $g'(0)=0$, and
$f=\mathcal A_ag$ solves $L*(f\,dx)=v$. The density is integrable and
$f(r)=O((a-r)^{-1/2})$ at the edge. Its $L$-potential is continuous
across $\partial B_a$, with value $v(a)$ there. For
$g^{(n)}=\sum_{j=0}^n\mathcal P_a^j\mathcal Tv$,
\begin{equation}\label{eq:error}
\|g-g^{(n)}\|_\infty\le
\frac{\omega(a)^{n+1}}{1-\omega(a)}\|\mathcal Tv\|_\infty.
\end{equation}
\end{lemma}

\begin{proof}
Let $J_0$ denote the Bessel function of the first kind and order zero.
Termwise integration of its power series gives
\[
\int_0^t\frac{rJ_0(\xi r)}{\sqrt{t^2-r^2}}\,dr
=\frac{\sin(\xi t)}{\xi},
\]
with value $t$ at $\xi=0$. Substituting the Abel formula for $f$,
Fubini's theorem and integration by parts therefore give
\[
\int_0^a f(r)J_0(\xi r)r\,dr=\int_0^a g(t)\cos(\xi t)\,dt,
\qquad
\mathcal T[J_0(\xi\,\cdot)](t)=\frac{\cos(\xi t)}{\pi^2}.
\]
To invert the Coulomb multiplier insert $e^{-\varepsilon\xi}$
and use
\[
\int_0^\infty e^{-\varepsilon\xi}J_0(r\xi)\cos(t\xi)\,d\xi
=\Re\bigl(r^2+(\varepsilon-it)^2\bigr)^{-1/2}.
\]

The square root has positive real part. The angular representation
of $J_0$ gives this Laplace transform.
As $\varepsilon\downarrow0$, its right-hand side tends to
$(r^2-t^2)^{-1/2}$ for $t<r$ and to zero for $t>r$.
Its modulus is bounded by $|r^2-t^2|^{-1/2}$, an integrable
majorant in $t$ for fixed $r>0$. Fubini's theorem for
$\varepsilon>0$ and dominated convergence therefore give
\[
\mathcal C*(f\,dx)(r)=2\pi\int_0^{\min(r,a)}
  \frac{g(t)}{\sqrt{r^2-t^2}}\,dt,\quad r>0.
\]
One integration by parts makes
the cosine integral $O(\xi^{-1})$. For fixed $r>0$,
$J_0(r\xi)=O(\xi^{-1/2})$, so the Hankel integral for
$L*(f\,dx)(r)$ converges absolutely. The same exponential
factor justifies its interchanges, and dominated convergence allows
$\varepsilon\downarrow0$.

The multiplier $\widehat L(\xi)=2\pi/\sqrt{1+\xi^2}$ then turns
$L*(f\,dx)=v$ into
\[
g(t)+\frac2\pi\int_0^a g(s)
\int_0^\infty\left(\frac{\xi}{\sqrt{1+\xi^2}}-1\right)
  \cos(\xi t)\cos(\xi s)\,d\xi\,ds=\mathcal Tv(t).
\]
The inner integral is absolutely convergent. The substitution
$\tau=\sin^2\theta$ leads to the identity
\[
1-\sqrt{\frac{z}{1+z}}
=\frac1\pi\int_0^1\frac{\sqrt{\tau/(1-\tau)}}{z+\tau}\,d\tau,\quad z>0,
\]
and with the cosine transform of $(\xi^2+\tau)^{-1}$ we get
\[
\int_0^\infty\left(1-\frac{\xi}{\sqrt{1+\xi^2}}\right)\cos(\xi x)\,d\xi
=h(x).
\]
With the product-to-sum identity, the integral equation becomes
$(I-\mathcal P_a)g=\mathcal Tv$.

The function $h$ is positive and even, strictly decreasing on
$[0,\infty)$, with $\int_{\R}h=\pi$. Thus
\[
\int_0^a[h(t-s)+h(t+s)]\,ds
=\int_{t-a}^{t+a}h(u)\,du
\le\int_{-a}^a h(u)\,du.
\]
Equality holds at $t=0$. Integrating the defining formula for $h$
proves \eqref{eq:contraction} and summing the geometric series bounds the remainder as in
\eqref{eq:error}.

The kernel $h$ is Lipschitz. Its derivative jumps by $-\pi/2$ at zero and
its distributional second derivative is a continuous function minus
$(\pi/2)\delta_0$, so convolution with the continuous even extension
of $g$ on $[-a,a]$ has two continuous derivatives on $(0,a)$
with one-sided limits at the endpoints, and by the equation
$g\in C^2[0,a]$ with $g'(0)=0$. Since $|g'(t)|\le Ct$, formula
\eqref{eq:abelinverse} bounds $|f(r)|$ by $C(a^2-r^2)^{-1/2}$
on $[0,a)$. Hence $f$ is integrable, with the asserted edge bound. By
Lemma~\ref{lem:trace} the $L$-potential of $f$ is finite and
continuous on $\R^2$, hence $u=L*(f\,dx)-v$ is continuous on $[0,a]$.
Split $L*(f\,dx)$ into $\mathcal C*(f\,dx)$ and a Hankel integral
with multiplier $\xi(1+\xi^2)^{-1/2}-1=O(\xi^{-2})$.
The Abel identity applies to the first term, and $\mathcal T$ passes
under the integral in the second. Therefore
$\mathcal T(L*(f\,dx))=(I-\mathcal P_a)g=\mathcal Tv$, and
$\int_0^t ru(r)(t^2-r^2)^{-1/2}\,dr$ has zero derivative on $(0,a)$;
it is at most $t\|u\|_\infty$ in modulus, hence vanishes identically.
Integrating it against $t(s^2-t^2)^{-1/2}$ over $0<t<s$ leaves
$\frac\pi2\int_0^s ru(r)\,dr=0$ for every $s\le a$, so $u=0$ on
$[0,a]$, and the $L$-potential equals $v$ on the closed disk, with
value $v(a)$ on its edge.
\end{proof}

\subsection{Positivity of densities}

We expand $\mathcal Tj_R$ in powers of $t$. For complex $r$ with $|r|<R$ the
squared distance $(R-re^{i\theta})(R-re^{-i\theta})$ from $re_1$ to the
circle stays off $(-\infty,0]$, so $j_R$ extends holomorphically to that
disk and its Taylor series converges on $[0,R)$. Averaging the Taylor
series of $L$ over circles leads, for $0\le t<R$, to
\begin{equation}\label{eq:series}
\pi^2(\mathcal Tj_R)(t)=\sum_{n=0}^\infty
  \frac{(\Delta^nL)(R)}{(2n)!}\,t^{2n}.
\end{equation}
The coefficient of $r^{2n}$ in $j_R(r)$ is
$(\Delta^nL)(R)/(4^n(n!)^2)$, and $\pi^2\mathcal T$ multiplies
that monomial by $4^n(n!)^2/(2n)!$. Every coefficient in
\eqref{eq:series} is positive, since
\[
\Delta(e^{-r}r^{-m})
=e^{-r}\bigl(r^{-m}+(2m-1)r^{-m-1}+m^2r^{-m-2}\bigr),
\quad m\ge1.
\]
Consequently $\pi^2(\mathcal Tj_R)(0)=e^{-R}/R$ and $(\mathcal Tj_R)'(t)>0$ for
$0<t<R$.

\begin{lemma}\label{lem:positivity}
Let $0<a<R$ and $k>0$. If
$(I-\mathcal P_a)g=\mathcal T(1-kj_R)$ and $f=\mathcal A_ag$, then
\[
f\ge0\text{ on }(0,a)\Longleftrightarrow g(a)\ge0.
\]
Under either condition, $-g'(t)>0$ for $0<t\le a$ and
$f(r)>0$ for every $0\le r<a$.
\end{lemma}

\begin{proof}
Put $v=1-kj_R$, and define
\[
d(t,s)=h(t-s)-h(t+s),\qquad
(\mathcal Q_au)(t)=\frac1\pi\int_0^a d(t,s)u(s)\,ds.
\]
The kernel is nonnegative and strictly positive for $t,s>0$, and
$\|\mathcal Q_a\|\le\omega(a)<1$. Differentiating $g=\mathcal Tv+\mathcal P_ag$ and integrating
by parts in $s$ yields
\begin{equation}\label{eq:derivative}
(I-\mathcal Q_a)(-g')(t)=-(\mathcal Tv)'(t)+\frac{d(t,a)}\pi g(a).
\end{equation}
Here $d(t,0)=0$ and
$\partial_t[h(t-s)+h(t+s)]=-\partial_s d(t,s)$.
Since $\mathcal T1=1/\pi^2$, $-(\mathcal Tv)'=k(\mathcal Tj_R)'>0$ on
$(0,a]$. If $g(a)\ge0$ the right-hand side is therefore positive for
$t>0$, and
the positive Neumann inverse of $I-\mathcal Q_a$ shows $-g'(t)>0$;
both terms in \eqref{eq:abelinverse} are then
nonnegative, the integral positive for every $r<a$.
Conversely
$\sqrt{a^2-r^2}\,f(r)\to g(a)$ as $r\uparrow a$.
So $f\ge0$ forces $g(a)\ge0$.
\end{proof}

For data $1 - kj_R$, linearity lets us work with the two component solutions. For each trial radius compute
\begin{equation}\label{eq:profiles}
\begin{aligned}
g_1&=(I-\mathcal P_a)^{-1}\mathcal T1,&
\phi_a&=\mathcal A_ag_1,\\
g_2&=(I-\mathcal P_a)^{-1}\mathcal Tj_R,&
\psi_{a,R}&=\mathcal A_ag_2,
\end{aligned}
\end{equation}

and set
\begin{equation}\label{eq:potentials}
F_a=K*(\phi_a\,dx),\qquad
H_{a,R}=K*\sigma_R-K*(\psi_{a,R}\,dx).
\end{equation}

On the disk, the discrepancy $F_a+kH_{a,R}-1$ solves $(1-\Delta)u=0$ and radiality with
regularity at the origin confines it to a multiple of $I_0$. Choosing $k$ for contact at the
center cancels this mode on the entire disk, leaving outer contact and nonnegativity
of the Abel edge coefficient to select $a$.

\begin{lemma}\label{lem:comparison}
For every $0<a<R$, $\phi_a>0$ in $B_a$,
$F_a(x)<1$, and $H_{a,R}(x)>0$ for all $x\in\R^2$.
\end{lemma}

\begin{proof}
The convolution of the indicators of two centered intervals is even
and nonincreasing on $[0,\infty)$, since it is their overlap length
after translation. Integrating the level-set representations extends
this conclusion to nonnegative even nonincreasing functions.
Apply that to the Neumann series for $g_1$, whose initial
term is $1/\pi^2$, extending each iterate by zero outside $[-a,a]$;
then $g_1$ is nonnegative and decreasing with $g_1(a)>0$, and the
Abel formula makes $\phi_a$ positive.
Its $L$-potential $J_\phi$ equals one on the disk and is continuous
at the edge. Outside the disk every source distance increases
strictly with the observation radius, hence $J_\phi<1$
there, and since the positive resolvent $(1-\Delta)^{-1}$ has a
strictly positive kernel of integral one,
$F_a=(1-\Delta)^{-1}J_\phi<1$ everywhere.

For the second sign apply Lemma~\ref{lem:lift} with
$\psi=\psi_{a,R}$. The planar trace $j_R-L*(\psi_{a,R}\,dx)$
is an integrable function, zero on the disk and positive off it, and
the strictly positive two-dimensional resolvent applied to it makes
$H_{a,R}$ positive.
\end{proof}

\subsection{Annular inequality}

\begin{lemma}\label{lem:annulus}
Let $0<a<R$ and let $\nu$ be a nonnegative finite radial measure
supported on $\overline B_a\cup\partial B_R$, with positive mass
on the outer circle. Put $U=K*\nu$ and $J=L*\nu$. Suppose
$U\in C^1$, $U=\gamma$ on $\overline B_a$, $U(R)=\gamma$,
and $J$ extends continuously from the annulus to $a$ with
$J(a)=\gamma$. Then $U(r)>\gamma$ for $a<r<R$.
After normalization, $\nu$ is the unique equilibrium measure
on $\overline B_R$.
\end{lemma}

\begin{proof}
Set $h=J-\gamma$ and $u=U-\gamma$. On the annulus
$\Delta J>0$ by \eqref{eq:plane}, so
$t\mapsto h(e^t)$ is strictly convex, with endpoint values
$h(a)=0$ and $h(R-)=+\infty$, the latter from the
positive outer-circle mass. Also
\[
(1-\Delta)u=h,\qquad u(a)=u'(a)=u(R)=0.
\]
Let $v$ solve $v''+v'/r-v=0$ with $v(R)=0$ and $v'(R)=-1$; the
one-dimensional maximum principle shows $v>0$ on $[a,R)$.
The Wronskian identity
\[
\bigl[r(vu'-v'u)\bigr]'=-rvh
\]
has integrable right-hand side, since $h$ has only a logarithmic
singularity at $R$. Integrating and using the endpoint data and the
finite trace of $U'$ gives $\int_a^R rvh\,dr=0$.
By strict convexity in $\log r$, $h$ is either positive
throughout or negative first and positive after one crossing, and
the zero integral rules out the first possibility. Every partial
integral $\int_a^r sv(s)h(s)\,ds$, $r<R$, is therefore negative.
So $(u/v)'>0$, and since $u(a)=0$ we get $u>0$ on the annulus.
Equality on the
support with the equilibrium inequality proves optimality after
normalization, and uniqueness follows from strict positive definiteness.
\end{proof}

\subsection{Equilibrium construction}

\begin{theorem}\label{thm:disk}
Fix $R>R_2$. For every trial $0<a<R$, form
\eqref{eq:profiles}--\eqref{eq:potentials} and put
\[
k(a,R)=\frac{1-F_a(0)}{H_{a,R}(0)}>0.
\]
There is a unique $a\in(0,R)$ satisfying
\begin{equation}\label{eq:selection}
F_a(R)+k(a,R)H_{a,R}(R)=1,\qquad
g_1(a)-k(a,R)g_2(a)\ge0.
\end{equation}
With $k=k(a,R)$ and $g=g_1-kg_2$, define
\begin{equation}\label{eq:diskmass}
M=k+2\pi\int_0^a g(t)\,dt.
\end{equation}
The maximizing probability measure and maximum diversity are
\begin{equation}\label{eq:diskmeasure}
\mu_R=\frac{\mathcal A_ag(|x|)}{M}\indicator{B_a}(x)\,dx
+\frac{k}{M}\sigma_R,\qquad
\Dmx{\overline B_R}=M.
\end{equation}
Its equilibrium level is $M^{-1}$, and $0<k/M<1$.
Its restriction to $B_R$ is absolutely continuous, and its density
is real analytic and strictly positive in $B_a$.
The Neumann approximations satisfy the
error bound \eqref{eq:error}.
\end{theorem}

\begin{proof}
By Lemma~\ref{lem:comparison} the denominator and numerator
defining $k$ are positive. On the central disk
$(1-\Delta)F_a=1$ and $(1-\Delta)H_{a,R}=0$, so as regular radial
solutions
\begin{equation}\label{eq:contactdisk}
F_a(r)=1+(F_a(0)-1)I_0(r),\qquad
H_{a,R}(r)=H_{a,R}(0)I_0(r).
\end{equation}
The definition of $k$ makes the potential one throughout the disk.

For existence take the unique equilibrium, and with
Proposition~\ref{prop:support} and \eqref{eq:density}
denote its inner radius, nonnegative density, outer mass and level by
$a_R,f_R,q_R,\gamma_R$; Proposition~\ref{prop:circle} ensures
$a_R>0$. On its central disk
\[
L*(f_R\,dx)=\gamma_R-q_Rj_R
\]
almost everywhere. The interior part of the measure has finite $L$-energy,
its nonnegative potential on the disk being bounded by
$\gamma_R$, and the Abel inverse of the same right-hand side
has finite absolute $L$-energy as well, since its absolute
density is bounded by a constant times
$(a_R^2-r^2)^{-1/2}$, whose unscreened potential on the disk
is the constant $\pi^2$. Both measures are absolutely continuous,
so the difference of the two interior measures, having zero
$L$-potential almost everywhere on the disk, has zero energy, and the
strictly positive multiplier $2\pi/\sqrt{1+|\xi|^2}$ forces that
difference to vanish. For these measures the Fourier energy identity
follows by truncating the positive Gaussian superposition
\eqref{eq:gauss} for $L$. Finite absolute energy permits passage to
the limit.
Thus
\[
\frac{f_R}{\gamma_R}
=\phi_{a_R}-\frac{q_R}{\gamma_R}\psi_{a_R,R}.
\]
Contact at the origin sets $q_R/\gamma_R$ equal to $k(a_R,R)$.
Outer contact proves the scalar equality in
\eqref{eq:selection}, and nonnegativity with the Abel edge limit proves its inequality.

Conversely, for a radius satisfying \eqref{eq:selection},
the density $f=\mathcal A_a(g_1-kg_2)$ is strictly positive in the
central disk by Lemma~\ref{lem:positivity}. The finite measure
$\nu=f\,dx+k\sigma_R$ has potential one on the disk by
\eqref{eq:contactdisk} and on the outer circle
by \eqref{eq:selection}, and by
Lemma~\ref{lem:abel} its $L$-potential
has trace one from both sides of the inner edge. Being atomless,
$\nu$ has $K*\nu\in C^1$ by dominated convergence applied to the
bounded gradient of $K$. Lemma~\ref{lem:annulus} then proves
the strict annular inequality and optimality after normalization.
Two selected radii would determine the same unique probability
measure, and the strictly positive interior densities force the
radii to coincide.

Integrating the Abel formula, we obtain
$\int_{B_a}f\,dx=2\pi\int_0^a g(t)\,dt$, hence
\eqref{eq:diskmass}. On its support the normalized potential is
$M^{-1}$, so the energy is $M^{-1}$ and the diversity is $M$.
\end{proof}

\subsection{Inner edge asymptotics}

\begin{theorem}\label{thm:edge}
Fix $R>R_2$ and let
\[
\mu_R=f_R(|x|)\indicator{B_a}(x)\,dx+q_R\sigma_R,
\qquad a=a_R,
\]
be the equilibrium probability measure from
Theorem~\ref{thm:disk}, and set $\gamma_R=I(\mu_R)$.
Let $g_R=(g_1-kg_2)/M$, with $k$ and $M$ from
Theorem~\ref{thm:disk}. The density has the Abel representation
\[
f_R(r)=\frac{g_R(a)}{\sqrt{a^2-r^2}}
  -\int_r^a\frac{g_R'(t)}{\sqrt{t^2-r^2}}\,dt.
\]
Then $g_R(a)>0$. With
$C_R=g_R(a)/\sqrt{2a}$, the potentials
$U_R=K*\mu_R$ and $J_R=L*\mu_R$ satisfy
\begin{align}
f_R(a-t)&=C_Rt^{-1/2}+O(t^{1/2}),\label{eq:densityedge}\\
J_R(a+s)&=\gamma_R-4\pi C_Rs^{1/2}+O(s^{3/2}),\label{eq:Jedge}\\
U_R(a+s)-\gamma_R&=\frac{16\pi C_R}{15}s^{5/2}
  +O(s^{7/2})\label{eq:Uedge}
\end{align}
as $s,t\downarrow0$.
\end{theorem}

\begin{proof}
By Lemma~\ref{lem:abel}, $g_R\in C^2[0,a]$ with one-sided
derivatives at the endpoints and $|f_R(r)|\le C(a^2-r^2)^{-1/2}$, and
Lemma~\ref{lem:positivity} shows $g_R(a)\ge0$ and $g_R'(a)<0$.
The exterior asymptotics come first, with the possibility
$g_R(a)=0$ kept open.

Put
\[
G=(L-\mathcal C)*(f_R\indicator{B_a})+q_Rj_R.
\]
Near $\partial B_a$ the function $G$ is $C^2$.
Here $L-\mathcal C=-1+|x|/2+O(|x|^2)$ near the origin, so its first
derivatives are bounded and its second derivatives are bounded by
$C/|x|$, while the density bound and \eqref{eq:nearfield},
applied with $f=f_R$, give, uniformly for small $\delta$,
\[
\sup_x\int_{|x-y|<\delta}
\frac{|f_R(y)|\indicator{B_a}(y)}{|x-y|}\,dy
\le C\delta^{1/2}.
\]
The second-derivative integrals with a continuous cutoff near the
origin are continuous. The displayed estimate bounds the omitted part
uniformly by $C\delta^{1/2}$, so their limit is continuous. The circle term is smooth
near $\partial B_a$, as $a<R$.

Extend $g_R$ to a $C^2$ function $\widetilde g$ on
$[0,a+\delta]$. By the unscreened Abel identity of
Lemma~\ref{lem:abel},
\[
\mathcal C*(f_R\indicator{B_a})(r)
=2\pi\int_0^{\min(r,a)}
\frac{g_R(t)}{\sqrt{r^2-t^2}}\,dt.
\]
Hence the $C^2$ function
\[
Z(r)=2\pi\int_0^{\pi/2}\widetilde g(r\sin\theta)\,d\theta+G(r)
\]
equals $\gamma_R$ for $r<a$, and for small $s>0$
\begin{equation}\label{eq:tail}
J_R(a+s)=Z(a+s)
-2\pi\int_a^{a+s}\frac{\widetilde g(t)}{\sqrt{(a+s)^2-t^2}}\,dt,
\qquad Z(a+s)=\gamma_R+o(s^2).
\end{equation}
If $g_R(a)=0$, substituting
$\widetilde g(t)=g_R'(a)(t-a)+O((t-a)^2)$ gives
\[
J_R(a+s)-\gamma_R
=-\frac{8\pi g_R'(a)}{3\sqrt{2a}}s^{3/2}
+o(s^{3/2}).
\]
The leading coefficient is positive. On the annulus
$u=U_R-\gamma_R$ satisfies
\[
u''+\frac1ru'-u=\gamma_R-J_R,
\qquad u(a)=u'(a)=0.
\]
Integrating twice would make $u(a+s)$ negative for small $s>0$,
contrary to Proposition~\ref{prop:support}.
So $g_R(a)>0$.

Now \eqref{eq:densityedge} follows from the Abel formula.
In \eqref{eq:tail}, the constant term in the expansion of
$\widetilde g$ contributes the leading term,
\[
2\pi\int_a^{a+s}\frac{\widetilde g(t)}{\sqrt{(a+s)^2-t^2}}\,dt
=\frac{4\pi g_R(a)}{\sqrt{2a}}s^{1/2}+O(s^{3/2}),
\]
and \eqref{eq:Jedge} follows.
Integrating the radial equation twice from $a$, with $u(a)=u'(a)=0$,
we obtain
\[
u(a+s)
=4\pi C_R\int_0^s(s-t)t^{1/2}\,dt+O(s^{7/2})
=\frac{16\pi C_R}{15}s^{5/2}+O(s^{7/2}).
\]
\end{proof}

\subsection{Threshold asymptotics}

Let $U_\partial(R)=\frac2\pi\int_0^{\pi/2}e^{-2R\sin t}\,dt$ be the
circle value of Proposition~\ref{prop:circle}, now as a function
of $R$, put $\mathcal F(R)=U_\partial(R)-e^{-R}$, which vanishes at
$R_2$, and
\[
d_*=1-e^{-R_2},
\qquad c_*=e^{-R_2}\left(1-\frac1{R_2}\right).
\]

\begin{theorem}\label{thm:threshold}
Let $m_R=\mu_R(B_{a_R})$ and define the probability $\nu_R$ on the
closed unit disk by
\[
\mu_R=m_R(T_{a_R})_\#\nu_R+(1-m_R)\sigma_R,
\qquad T_a(z)=az.
\]
Then $R_2>1$, $\mathcal F'(R_2)>0$, and, as
$\varepsilon=R-R_2\downarrow0$,
\begin{align}
m_R&=\frac{\mathcal F'(R_2)}{d_*}\varepsilon
  +o(\varepsilon),\label{eq:masslimit}\\
a_R&=\frac{\pi\mathcal F'(R_2)}{2c_*d_*}\varepsilon
  +o(\varepsilon),\label{eq:radiuslimit}\\
\nu_R&\rightharpoonup\mu_{\mathrm C}.\label{eq:limitlaw}
\end{align}
The equilibrium level and maximum diversity satisfy
\begin{align}
\gamma_R&=U_\partial(R)-\frac{\mathcal F(R)^2}{d_*}
+o(\varepsilon^2),\label{eq:energylimit}\\
\Dmx{\overline B_R}-\frac1{U_\partial(R)}
&=\frac{e^{2R_2}\mathcal F'(R_2)^2}{d_*}
\varepsilon^2+o(\varepsilon^2).
\label{eq:diversitylimit}
\end{align}
\end{theorem}

\begin{proof}
The two chord bounds
\[
\sin t\ge\frac{3t}{\pi}
\quad(0\le t\le\pi/6),\qquad
\sin t\ge\frac14+\frac{3t}{2\pi}
\quad(\pi/6\le t\le\pi/2)
\]
give
\[
U_\partial(1)<\frac{1+e^{-1}-2e^{-2}}3<e^{-1}.
\]
The last inequality uses $e<1+\sqrt3$.
The unique positive crossing from
Proposition~\ref{prop:circle} thus satisfies $R_2>1$.
Also $\chi(R)=e^RU_\partial(R)$ is strictly convex, with
$\chi(0)=\chi(R_2)=1$ and $\chi'(0)<0$, whence
$\chi'(R_2)>0$ and
$\mathcal F'(R_2)=e^{-R_2}\chi'(R_2)>0$.
In particular $c_*,d_*>0$.

Let $R_j\downarrow R_2$. All $\mu_{R_j}$ are supported in a fixed
compact disk, so a subsequence converges weakly to a probability
$\eta$ supported on $\overline B_{R_2}$. For every probability
$\nu$ on $\overline B_{R_2}$, let $T_jx=(R_j/R_2)x$.
Minimality and continuity of the energy kernel give
\[
I(\eta)=\lim_j I(\mu_{R_j})
\le\lim_j I((T_j)_\#\nu)=I(\nu).
\]
Uniqueness at $R_2$ implies $\eta=\sigma_{R_2}$. Every
subsequential limit is therefore the same, so
$\mu_R\rightharpoonup\sigma_{R_2}$ as $R\downarrow R_2$.
By their common Lipschitz bound the potentials converge uniformly on
compact sets. Hence $\gamma_R\to e^{-R_2}$.

The limiting circle potential has the expansion
\[
K*\sigma_{R_2}(r)
=e^{-R_2}+\frac{c_*}{4}r^2+O(r^4).
\]
A positive subsequential limit of $a_R$ would make this potential
constant on a nonempty disk, contradicting $c_*>0$.
So $a_R\to0$. With the central disk shrinking to the origin, weak
convergence to $\sigma_{R_2}$ implies $m_R\to0$ too.

By Lemma~\ref{lem:abel} the identity $J_R=\gamma_R$ holds
pointwise on the central disk, and since
$|e^{-u}-1|\le u$ and the circle's $L$-potential is smooth and radial
near the origin, it gives, uniformly for $|z|<1$,
\begin{equation}\label{eq:coulomblimit}
\frac{m_R}{a_R}(\mathcal C*\nu_R)(z)
=\gamma_R-(1-m_R)\frac{e^{-R}}R+O(m_R+a_R^2).
\end{equation}
Integrate against $\mu_{\mathrm C}$, using
Lemma~\ref{lem:coulomb} and Tonelli's theorem, to get
\[
\frac\pi2\frac{m_R}{a_R}
=\gamma_R-(1-m_R)\frac{e^{-R}}R+O(m_R+a_R^2)
\longrightarrow c_*.
\]
Therefore $m_R/a_R\to2c_*/\pi$, and
by \eqref{eq:coulomblimit}
$\mathcal C*\nu_R\to\pi/2$ uniformly on the open unit disk.
The measures $\nu_R$ have zero mass on its boundary, so
$E_{\mathrm C}(\nu_R)\to\pi/2$. Lower semicontinuity of Coulomb
energy \cite[Lemma~2.2.1(e)]{fuglede1960} and uniqueness in
Lemma~\ref{lem:coulomb} imply that every weak limit equals
$\mu_{\mathrm C}$. That is \eqref{eq:limitlaw}.

For the rates let
\[
A_R=\int e^{-a_R|z|}\,d\nu_R(z),\qquad
X_R=\int (K*\sigma_R)(a_R|z|)\,d\nu_R(z).
\]
The contact equations at the center and on the outer circle yield
\[
(1-m_R)\mathcal F(R)=m_R(A_R-X_R).
\]
Here $A_R=1+O(a_R)$ and $X_R=e^{-R}+O(a_R^2)$,
so $m_R/\mathcal F(R)\to1/d_*$, and the simple zero of $\mathcal F$
with the limit of $m_R/a_R$ proves
\eqref{eq:masslimit} and \eqref{eq:radiuslimit}.

Finally put
\[
Y_R=\iint e^{-a_R|z-w|}\,d\nu_R(z)\,d\nu_R(w)=1+O(a_R).
\]
The energy decomposes as
\[
\gamma_R=(1-m_R)^2U_\partial(R)
+2m_R(1-m_R)X_R+m_R^2Y_R.
\]
Since $a_R,m_R=O(\varepsilon)$, it becomes
\[
\gamma_R=U_\partial(R)-2m_R\mathcal F(R)
+m_R^2[1-2e^{-R}+U_\partial(R)]+O(\varepsilon^3).
\]

Substituting $m_R=\mathcal F(R)/d_*+o(\varepsilon)$ and
$1-2e^{-R_2}+U_\partial(R_2)=d_*$ gives
\eqref{eq:energylimit}, and taking reciprocals, with
$U_\partial(R_2)=e^{-R_2}$, leads to
\eqref{eq:diversitylimit}.
\end{proof}

Together with the boundary formula in
Proposition~\ref{prop:circle}, \eqref{eq:diversitylimit}
shows that $\Dmx{\overline B_R}$ is differentiable at $R_2$.
Its difference from the circle magnitude begins at second order
above the threshold.

\section{Discussion}\label{sec:disc}

How many interior support components can occur in dimensions $N\ge4$?
The outer sphere still carries positive mass, with an open annulus
separating it from any interior support (Lemma~\ref{lem:boundary}). We had hoped to extend the same support argument to higher dimensions.
The interior contact bound $J\le\gamma$ persists
(Lemma~\ref{lem:contact}), but $\Delta L$ in
\eqref{eq:laplacian} changes sign for $N\ge4$. The kernel calculation therefore no longer provides the subharmonicity used above to exclude mass-free regions between interior contact radii.
See Appendix~\ref{app:tables} for numerical experiments performed for $N=4,5$.

Maximum diversity is continuous in Gromov--Hausdorff distance on compact
positive definite metric spaces \cite{meckes2013}.
For a nonempty compact set $E\subset\R^N$ and a finite $h$-net
$E_h\subset E$,
\[
\Dmx{E_h}\le\Dmx{E}\le e^{2h}\Dmx{E_h}.
\]
A measurable nearest point map $T:E\to E_h$ satisfies
$|Tx-Ty|\ge |x-y|-2h$, so pushing a probability measure forward
increases its energy by at most a factor $e^{2h}$.
Minimizing and taking reciprocals proves the right inequality, and
inclusion proves the left.
With rigorous bounds for the finite optimization the estimate bounds
the maximum diversity value. Support transitions still require control
of the potential between the sampled points.

Leinster and Roff also ask about cubes.
On a box with the $\ell_1$ metric the exponential kernel factorizes,
and the positive weight measure is a product of interval weight
measures \cite{leinstermeckes2017survey}.
For the Euclidean metric, the proof of
Lemma~\ref{lem:contact} uses only that $U\ge\gamma$ near the
contact point, so $J\le\gamma$ at interior contact on any compact set,
and for $N=2,3$ the function $J$ is strictly subharmonic off of the support
whatever the set. The cube lacks the radial control of $J$ up to the boundary of a
mass-free region (Lemma~\ref{lem:traces}) and the radial monotonicity that
puts the whole boundary in the support (Lemma~\ref{lem:boundary}),
and its support is yet to be determined.

Does the interior support component remain separated from the outer boundary under
a small smooth perturbation of a three-dimensional ball with
$R>R_3$? A perturbed construction would require constant
potential on an unknown interior region and nonnegative surface masses on both
boundaries, with the equilibrium inequality between them. The potential excess
must be controlled up to the moving inner boundary, where it vanishes.

In the planar construction the contraction factor tends to one as the
inner radius grows. Direct Neumann iteration gets slow on large
disks, and the inverse square root edge law suggests putting the singular
factor into the numerical basis. Quantitative error bounds for the
quadrature and scalar root are needed as the inner radius grows.
The planar large-radius asymptotics also remain to be determined.

\backmatter

\section*{Declarations}

\subsection*{Funding}
No external funding was received for this work.

\subsection*{Competing interests}
The author declares no competing interests.

\subsection*{Ethics approval and consent to participate}
Not applicable.

\subsection*{Consent for publication}
Not applicable.

\subsection*{Data availability}
No external datasets were used.

\subsection*{Materials availability}
Not applicable.

\subsection*{Code availability}
Available upon request.

\subsection*{Author contribution}
Adam Dajani is the sole author and is responsible for the research and manuscript.

\subsection*{Use of generative AI tools}
No generative AI was used in this work.

\begin{appendices}
\section{Proof of analyticity}\label{app:analyticity}

\begin{proof}[Proof of Lemma~\ref{lem:analyticity}]
Set $g=H-c$, a tempered distribution given by a locally
integrable function vanishing almost everywhere in $B_a$; since
$\mathcal B(c)=c$ it suffices to compute $\mathcal Bg$ there.
Against a test function compactly supported in $B_a$ the first term of
\eqref{eq:semigroup} pairs to zero with $g$, and the test support
has positive distance from $\R^2\setminus B_a$, so Fubini in the
remaining heat integral yields
\[
\mathcal Bg(x)=-\int_{\R^2\setminus B_a}g(y)n(x-y)\,dy.
\]
The interchange is absolutely convergent because $n$
is bounded for $|z|\ge\delta>0$ and has an integrable exponential
tail, and $|g(y)|\le |H(y)|+|c|$.

For analyticity we extend this integral to a complex neighborhood.
Fix $x_0\in B_a$ and choose $\delta>0$ with
$\dist(x_0,\R^2\setminus B_a)\ge\delta$.
For $z\in\C^2$ with $|z-x_0|<\delta/4$ and real $y\notin B_a$
use the principal square root in
$d(z,y)=\sqrt{(z-y)\cdot(z-y)}$, with the complex bilinear dot product.
With $u=x_0-y$ the expression under the root differs from $|u|^2$
by less than $(9/16)|u|^2$ and stays in the right half-plane, and the
same estimate yields $\Re d(z,y)\ge (\sqrt7/4)|u|$ and
$|d(z,y)|\ge(\sqrt7/4)|u|$.
Multiplied by $|H(y)|+|c|$, the holomorphic kernel
$(1+d)e^{-d}/(2\pi d^3)$ has an integrable majorant, uniform on every
smaller complex ball. The integral in
\eqref{eq:analyticity} is consequently holomorphic there. Its
restriction to the real ball is real analytic.
\end{proof}

\section{Numerical values and experiments}\label{app:tables}

The left table is for $\R^3$
(Remark~\ref{rem:limits}) and the right one for the plane. All values in those two tables are rounded to four decimal
places. The planar values use
Chebyshev--Nystr\"om discretization of \eqref{eq:profiles} with 24, 36 and 48 nodes, graded angular quadrature, and a bracketed
solve of \eqref{eq:selection}. The computed potential agrees with $\gamma$ at five sampled central radii and on the outer circle, and exceeds it at nineteen sampled interior annular radii.

\begin{center}
\small
\begin{tabular}[t]{@{}crrr@{}}
\multicolumn{4}{@{}c@{}}{$N=3$}\\
\midrule
$R$ & $b$ & $\Dmx{\overline B_R}$ & $\Mag(\overline B_R)$\\
\midrule
3 & 2.4720 & 18.3297 & 20.5000\\
4 & 2.1334 & 32.6738 & 35.6667\\
5 & 1.9931 & 52.8648 & 56.8333\\
6 & 1.9115 & 79.9081 & 85.0000\\
7 & 1.8573 & 114.8054 & 121.1667\\
8 & 1.8184 & 158.5576 & 166.3333\\
9 & 1.7890 & 212.1648 & 221.5000\\
10 & 1.7659 & 276.6275 & 287.6667\\
\botrule
\end{tabular}\hspace{2.5em}%
\begin{tabular}[t]{@{}crrr@{}}
\multicolumn{4}{@{}c@{}}{$N=2$}\\
\midrule
$R$ & $a_R$ & $\Dmx{\overline B_R}$ & $\mu_R(\partial B_R)$\\
\midrule
1.4 & 0.0720 & 4.0153 & 0.9965\\
1.5 & 0.2396 & 4.3106 & 0.9855\\
2 & 0.8868 & 5.9325 & 0.9175\\
3 & 1.9861 & 9.9155 & 0.7829\\
\botrule

\end{tabular}
\end{center}

\begin{center}
\small
\begin{tabular}{@{}lrr@{}}

\multicolumn{3}{@{}c@{}}{Numerical testing for $N = 4,5$}\\
\midrule
Tested quantity & $N=4$ & $N=5$\\
\midrule
First interior sphere (boundary threshold $R_N$) & 4.103 & 5.483\\
Second interior sphere & 4.7925 & 6.8032\\
Third interior sphere & 4.9081 & 7.2158\\
Fourth interior sphere & 4.9272 & 7.3493\\
Fifth interior sphere & 4.9303 & 7.3922\\
Sixth interior sphere & -- & 7.4059\\

Apparent accumulation & 4.931 & 7.412\\
The shrink factor of subsequent birth intervals & 6 & 3.1\\
\botrule
\end{tabular}
\end{center}

For $N=4,5$ we solved the contact and stationarity equations for
finitely many spheres and followed the solutions in $R$.
Potential excess was evaluated at 1,000 radial samples and minimized
locally near sampled minima, with acceptance tolerance $-10^{-15}$.
The computations indicate that boundary measure is optimal up to the thresholds in the table,
the radii at which the potential of $\sigma_R$ at
the center, $e^{-R}$, falls to its value $U_\partial(R)$ on
$\partial B_R$. Just above each threshold the interior mass sits on a
single sphere whose radius grows from zero.

Further spheres appear at the center, and a minimization on a grid of 800 radii, without continuation, finds the
same number of spheres between successive birth radii. The intervals between successive birth radii shrink by the factors in the table, so the births appear to accumulate near the
listed accumulation radii. Up to those radii the computed support is a finite union of concentric spheres whose number grows with the radius. At larger radii the discrete minimizers spread mass over a central region. We were unable to distinguish concentric spheres from a central ball.

\end{appendices}

\bibliography{references}

\end{document}